\documentclass[12pt,a4paper]{amsart}

\usepackage[T1]{fontenc}
\usepackage[utf8]{inputenc}
\usepackage{geometry}
\usepackage{amsmath,amssymb,amsthm}
\usepackage{cite}
\usepackage{footmisc}

\numberwithin{equation}{section}

\makeatletter
\renewenvironment{proof}[1][Proof]{%
  \par\pushQED{\qed}%
  \normalfont\topsep6\p@\@plus6\p@\relax
  \trivlist
  \item[\hskip\labelsep\bfseries #1:]\ignorespaces
}{%
  \popQED\endtrivlist\@endpefalse
}
\makeatother

\theoremstyle{plain}
  \newtheorem{Thm}{Theorem}[section]
\newtheorem{lem}[Thm]{Lemma}
 \newtheorem{prop}[Thm]{Proposition}
\newtheorem{Cor}[Thm]{Corollary}
\newtheorem{conj}[Thm]{Conjecture}
\theoremstyle{definition}
\newtheorem{Def}[Thm]{Definition}

\newtheorem{rem}[Thm]{Remark}

\begin{document}

\title{The derived set of multiple star polylogarithms}

\author{Jiangtao Li}

\email{lijiangtao@csu.edu.cn}
\address{Jiangtao Li \\ School of Mathematics and Statistics, HNP-LAMA, Central South University, Hunan Province, China}

\begin{abstract}
The derived set of multiple zeta-star values is the half-line $[1,+\infty)$.  In this paper we study the corresponding two-dimensional problem for multiple star polylogarithms on the unit circle.  We first prove that every shifted multiple star polylogarithm is the generating function of a completely monotone sequence and hence admits a normalized Hausdorff--Stieltjes representation.  Both the shifted and the unshifted functions, as well as their infinite-depth limits, are shown to be univalent on the half-plane $\mathrm{Re}\,z<1$.

For finite indices, the representing densities satisfy a strict monotone likelihood-ratio order with respect to the reverse lexicographic order.  Together with closure properties of the Hausdorff--Stieltjes class and a limiting quotient argument, this shows that, along the upper semicircle, the argument of each translated curve increases strictly while its modulus decreases strictly; the corresponding lower-semicircle statement follows by conjugation.  The same idea gives strict separation of the curves attached to different indices.  We then establish a one-to-one correspondence between a natural set of pairs consisting of a point of the unit circle and an infinite index, and the closed half-plane $\mathrm{Re}\,w\geq \frac12$ with the point $1$ removed.  Under a binary coordinate on the index set, this correspondence is a homeomorphism.

As an application, a subclass of shifted cyclotomic multiple zeta-star values of all levels form a countable dense subset of the open half-plane $\mathrm{Re}\,w>\frac12$.  Consequently, the derived set, and every higher derived set, of this cyclotomic family is the closed half-plane $\mathrm{Re}\,w\geq\frac12$.
\end{abstract}

\let\thefootnote\relax\footnotetext{
2020 $\mathnormal{Mathematics} \;\mathnormal{Subject}\;\mathnormal{Classification}$: 11M32, 11G55, 30C45.\\
$\mathnormal{Keywords:}$ Multiple zeta-star values, multiple polylogarithms, Hausdorff moment sequences, Stieltjes functions, univalent functions.\\
}

\maketitle

\section{Introduction}\label{int}
Multiple zeta values and multiple zeta-star values are defined respectively by
\[
\zeta(k_1,\ldots,k_r)=\sum_{n_1>\cdots>n_r\geq 1}
\frac{1}{n_1^{k_1}\cdots n_r^{k_r}},
\]
\[
\zeta^{\star}(k_1,\ldots,k_r)=\sum_{n_1\geq\cdots\geq n_r\geq 1}
\frac{1}{n_1^{k_1}\cdots n_r^{k_r}},
\qquad k_1\geq 2,\quad k_2,\ldots,k_r\geq 1.
\]
Their weight and depth are $k_1+\cdots+k_r$ and $r$, respectively.  These values occur naturally in the study of iterated integrals, multiple series, mixed Tate motives, algebraic $K$-theory; see, for example, \cite{gonc,chen,DG,brow,ganl,gon1,gon2,ikz,zag}.   They also arise in perturbative quantum field theory and in the evaluation of Feynman integrals; see \cite{bro1,duhr,pan}.
 In addition to their algebraic relations, it is natural to ask for the order, topology, and metric geometry of the sets formed by these special values.

Denote by $\mathcal{Z}^{\star}$ the set of multiple zeta-star values.  In \cite{lit}, the author obtained a natural total order on $\mathcal{Z}^{\star}$ from multiple integral representations. Denote by $\mathbb{Z}^+$ the set of positive integers. Define
\[
\mathcal{T}=\bigg{\{}(k_1,\cdots, k_r,\ldots)\in (\mathbb{Z}^+)^{\infty}\,\bigg{|}\,
k_1\geq 2, (k_1,\cdots,k_r,\cdots)\neq (2,\{1\}^{\infty})\bigg{\}}.
\]
Then the map
\[
\eta:\mathcal{T}\rightarrow(1,+\infty),
\]
\[
(k_1,\cdots,k_r,\ldots)\mapsto
\lim_{r\rightarrow+\infty}\zeta^{\star}(k_1,\ldots,k_r)
\]
is bijective.  This is Theorem 1.3 of \cite{lit}. We call the map $\eta$ zeta-star correspondence. An equivalent correspondence was found independently by Hirose, Murahara and Onozuka in \cite{hmo}.  It implies that $\mathcal{Z}^{\star}$ is dense in $(1,+\infty)$ and that its derived set is $[1,+\infty)$.  The result may be compared with continued fractions: an infinite symbolic index provides a coordinate for an ordinary real interval.  It has since led to rational deformations, Diophantine approximation and normalized approximation function,
restricted sumsets and product sets, finite correspondences, and complex
analytic variations
\cite{LiRational2025,Li2025,LiYang2026,LiDivided2026,Li2026}. Kamano \cite{Kamano2026} investigated the order structure of  a class of multi- polylogarithm functions.

In the present paper this one-dimensional bijection is used only when both unit-circle parameters equal $1$. For $|z|<1$, the multiple polylogarithm and the multiple star polylogarithm are
\[
\mathrm{Li}_{k_1,\ldots,k_r}(z)=
\sum_{n_1>\cdots>n_r\geq 1}
\frac{z^{n_1}}{n_1^{k_1}\cdots n_r^{k_r}},
\]
\[
\mathrm{Li}^{\star}_{k_1,\ldots,k_r}(z)=
\sum_{n_1\geq\cdots\geq n_r\geq 1}
\frac{z^{n_1}}{n_1^{k_1}\cdots n_r^{k_r}}.
\]
If $k_1\geq 2$, these series are absolutely convergent on $|z|=1$.  If $k_1=1$ and $|z|=1$, $z\neq 1$, their boundary values are defined by partial sums with respect to the outer variable $n_1$. Section \ref{sec:cm} shows that these sums agree with the analytic continuations.  Denote by 
\[
\widetilde{\mathrm{Li}}^{\star}_{k_1,\ldots,k_r}(z)
=\frac{\mathrm{Li}^{\star}_{k_1,\ldots,k_r}(z)}{z}.
\]
We call it the shifted multiple star polylogarithm.

Multiple polylogarithms (MPLs) constitute a broad family of multivalued transcendental functions that simultaneously generalize the classical polylogarithm \(\operatorname{Li}_n(z)\) and multiple zeta values. Although they may be introduced through nested power series, their most flexible formulation is in terms of Chen iterated integrals of logarithmic differential forms on punctured projective lines. This representation makes their analytic continuation, differential equations, monodromy, and algebraic structure particularly transparent \cite{chen,gon1}. MPLs are commonly organized by their weight and depth and satisfy rich systems of functional identities. In particular, the iterated-integral representation gives rise to shuffle relations, whereas the nested-series representation produces stuffle, or quasi-shuffle, relations; the compatibility of these two algebraic structures plays a central role in the study of multiple zeta values and their generalizations \cite{fur,gon1,gon2}.  The p-adic iterated integration of MPLs is fundamental to the theory of p-adic multiple zeta values \cite{col,fur}.  Goncharov’s Hopf-algebraic and motivic formulation relates MPLs to periods of mixed Tate structures, motivic fundamental groups, and Galois actions, while Brown’s results demonstrate the effectiveness of this framework in describing the structure of motivic multiple zeta values \cite{brow,gon1,gon2}. Beyond arithmetic geometry and number theory, MPLs have become an essential function space in perturbative quantum field theory, since broad classes of multiloop Feynman integrals and scattering amplitudes can be expressed, simplified, and evaluated in terms of MPLs, their symbols, coproducts, and associated hyperlogarithmic integration algorithms \cite{bro1,duhr,pan}. These connections motivate continuing research on functional identities, bases and algebraic independence, special values, single-valued and motivic versions, analytic continuation, and efficient symbolic and numerical evaluation.

The passage from the real interval to the unit circle introduces a geometric function-theoretic problem.  The coefficients of a shifted multiple star polylogarithm form a Hausdorff moment sequence.  Thus the function belongs to the normalized Hausdorff--Stieltjes class
\[
h(z)=\int_{[0,1]}\frac{d\mu(t)}{1-tz},
\qquad \mu([0,1])=1.
\]  This connects the subject with completely monotone sequences, Pick functions, and univalent functions; see \cite{hau,wid,wir,Donoghue1974,GK}.  A central point of the present paper is that the representing measures are not merely ordered: their densities satisfy a strict monotone likelihood-ratio order.  This strictness is what prevents two different index curves from intersection on the unit circle.

Let
\[
\mathcal{D}=\{z\in\mathbb{C}\,|\,\mathrm{Re}\,z<1\},
\qquad
\Lambda=\mathbb{C}\setminus[1,+\infty).
\]
Our first main result is the following.

\begin{Thm}\label{cm}
Let $k_1,\ldots,k_r\geq 1$.

$(i)$ The shifted multiple star polylogarithm
\[
\widetilde{\mathrm{Li}}^{\star}_{k_1,\ldots,k_r}(z)
=\sum_{n=0}^{+\infty}A_{k_1,\ldots,k_r}(n)z^n
\]
is the generating function of a completely monotone sequence with $A_{k_1,\ldots,k_r}(0)=1$.  It admits a Hausdorff--Stieltjes representation on $\Lambda$.

$(ii)$ The multiple star polylogarithm $\mathrm{Li}^{\star}_{k_1,\ldots,k_r}$ extends holomorphically to $\Lambda$ and is univalent on $\mathcal{D}$.

$(iii)$ The shifted multiple star polylogarithm $\widetilde{\mathrm{Li}}^{\star}_{k_1,\ldots,k_r}$ extends holomorphically to $\Lambda$ and is univalent on $\mathcal{D}$.

$(iv)$ For every infinite index ${\bf k}=(k_1,k_2,\ldots)\in(\mathbb{Z}^{+})^{\infty}$, the limits
\[
\mathrm{Li}^{\star}_{\bf k}(z)
=\lim_{r\rightarrow+\infty}\mathrm{Li}^{\star}_{k_1,\ldots,k_r}(z),
\qquad
\widetilde{\mathrm{Li}}^{\star}_{\bf k}(z)
=\lim_{r\rightarrow+\infty}\widetilde{\mathrm{Li}}^{\star}_{k_1,\ldots,k_r}(z)
\]
exist locally uniformly on $\mathcal{D}$.  Both limit functions are univalent on $\mathcal{D}$, and the shifted limit has a Hausdorff--Stieltjes representation on $\Lambda$.
\end{Thm}

Put
\[
S^1=\{z\in\mathbb{C}\,|\,|z|=1\},
\qquad
\widehat{\mathcal{T}}=(\mathbb{Z}^{+})^{\infty}.
\]
Define
\[
\widetilde{\mathcal{ST}}=
\Big\{(z,{\bf k})\in S^1\times\widehat{\mathcal{T}}\,\Big|\,
\text{if }z=1,\text{ then }k_1\geq 2
\text{ and }{\bf k}\neq(2,1,1,\ldots)\Big\}.
\]
Thus, when $z=1$, precisely the indices for which the infinite zeta-star limit is finite are retained.  Let
\[
\mathcal{H}=\left\{w\in\mathbb{C}\,\bigg|\,\mathrm{Re}\,w\geq\frac12\right\},
\qquad
\mathcal{H}^{o}=\left\{w\in\mathbb{C}\,\bigg|\,\mathrm{Re}\,w>\frac12\right\}.
\]
The right half complex plane \(\mathcal{H}\)  plays a fundamental role in the theory of Riemann zeta function. The principal result is a two-dimensional extension of the zeta-star correspondence.

\begin{Thm}\label{one}
The map
\[
\eta_{\mathcal{S}}:\widetilde{\mathcal{ST}}\longrightarrow\mathcal{H}\setminus\{1\},
\qquad
(z,{\bf k})\longmapsto
\widetilde{\mathrm{Li}}^{\star}_{\bf k}(z)
\]
is bijective.
Moreover, put
\[
\beta({\bf k})=\frac{1}{2^{k_1}}+\cdots+\frac{1}{2^{k_1+\cdots +k_r}}+\cdots.
\]
Equip $\widehat{\mathcal{T}}$ with the topology transported from $(0,1]$ by $\beta$, and equip $\widetilde{\mathcal{ST}}$ with the corresponding subspace topology of $S^1\times\widehat{\mathcal{T}}$.  Then
\[
\tau_{\mathcal{S}}(z,{\bf k})=z\beta({\bf k})
\]
identifies $\widetilde{\mathcal{ST}}$ with
\[
\overline{\mathcal{U}}\setminus\left(\{0\}\cup\left[1/2,1\right]\right),
\]
and $\eta_{\mathcal{S}}\circ\tau_{\mathcal{S}}^{-1}$ is a homeomorphism onto $\mathcal{H}\setminus\{1\}$.
Here 
\[
\mathcal{U}=\{u\in\mathbb{C}\,|\,|u|< 1\},\quad \overline{\mathcal{U}}=\{u\in\mathbb{C}\,|\,|u|\leq 1\},
\]
\end{Thm}

Putting $z=1$ recovers the zeta-star correspondence.  Putting $z=-1$ gives a second real slice.

\begin{Cor}\label{minusone}
The map
\[
\eta_{-1}:\widehat{\mathcal{T}}\longrightarrow\left[\frac12,1\right),
\qquad
{\bf k}\longmapsto
\lim_{r\rightarrow+\infty}
\sum_{n_1\geq\cdots\geq n_r\geq 1}
\frac{(-1)^{n_1-1}}{n_1^{k_1}\cdots n_r^{k_r}}
\]
is bijective.
\end{Cor}

For $N\geq 1$, let $\mu_N$ be the set of $N$-th roots of unity and define
\[
\mathcal{C}^{\star}=
\bigcup_{N\geq1}\left\{
\sum_{n_1\geq\cdots\geq n_r\geq 1}
\frac{z^{n_1-1}}{n_1^{k_1}\cdots n_r^{k_r}}
\,\bigg|\,
 z\in\mu_N,\ r\geq 1,\ k_1,\ldots,k_r\geq1,\ (k_1,z)\neq(1,1)
\right\}.
\]
For a subset $E\subseteq\mathbb{C}$, denote its derived set by $E^{\prime}$, put $E^{(0)}=E$, and define $E^{(n+1)}=(E^{(n)})^{\prime}$.

\begin{Thm}\label{cyc}
The set $\mathcal{C}^{\star}$ is a countable dense subset of $\mathcal{H}^{o}$.  Consequently,
\[
(\mathcal{C}^{\star})^{(n)}=\mathcal{H},
\qquad n\geq 1.
\]
\end{Thm}

We call the bijective map $\eta_{\mathcal{S}}$ polylogarithms star correspondence. The polylogarithms star correspondence is a generalization of original zeta star correspondence.   The correspondence has several potential applications.  First, it gives a symbolic coordinate system for the punctured half-plane, in which the radial coordinate is encoded by an infinite positive-integer sequence and the angular coordinate is encoded by a point of the unit circle.  Secondly, Theorem \ref{cyc} provides a constructive approximation principle: every point of $\mathcal{H}^{o}$ can be approximated by cyclotomic multiple zeta-star values.  This may be useful in experimental arithmetic and in the numerical study of special values at roots of unity.  Thirdly, the Hausdorff--Stieltjes representation gives positivity, moment inequalities, and Pick-function structure, which can be used for coefficient estimates and boundary-value problems.  These applications are geometric and approximation-theoretic. The occurrence of the half-plane $\mathrm{Re}\,w\geq\frac12$ by itself does not imply a statement about zeros of the Riemann zeta function.

 \begin{rem}
  In a remarkable paper \cite{lew}, J. Lewis studied the polylogarithms 
  \[
  \mathrm{Li}_{\alpha}=\sum_{n=1}^{+\infty} \frac{z^n}{n^{\alpha}},\; \alpha \geq 0
  \]
  for $|z|<1$. Lewis showed that $ \mathrm{Li}_{\alpha}$ maps the open unit disc one-to-one onto a convex domain for $\alpha\geq 0$.  For $k_1,\cdots, k_r\geq 1$, as the multiple star polylogarithm 
      \[
{ \mathrm{Li}^\star_{k_1,\cdots,k_r}(z)}=\sum_{n_1\geq  \cdots\geq  n_r\geq 1}\frac{z^{n_1}}{n_1^{k_1}\cdots n_r^{k_r}}
              \] 
               is univalent on the domain  \[\mathcal{D}=\big{\{}z\;\big{|}\, z\in \mathbb{C}, \mathrm{Re}\; (z)<1\big{\}} ,\] 
               can we show that  the multiple star polylogarithm 
               \[
 \mathrm{Li}^\star_{k_1,\cdots,k_r}(z)=\sum_{n_1\geq  \cdots\geq  n_r\geq 1}\frac{z^{n_1}}{n_1^{k_1}\cdots n_r^{k_r}}\] maps the domain $\mathcal{D}$ one-to-one onto a convex domain?
                 \end{rem}
                 
         For two holomorphic functions $f$ and $g$ on $\mathcal{U}$. We say that $g$ is subordinate to $f$ $(f\succ g)$ if there exists a       
                 holomorphic map $p: \mathcal{U}\rightarrow \mathcal{U}$ such that
                 \[
                 f(p(z))=g(z),\quad p(0)=0.
                 \]
               For positive multi-indices
$\mathbf k=(k_1,\ldots,k_r)$ and
$\boldsymbol\ell=(\ell_1,\ldots,\ell_s)$, write
$\mathbf k\succ\boldsymbol\ell$ when either $\boldsymbol\ell$ is a proper prefix of $\mathbf k$, or
 at the first position $i$ at which they differ, one has
       $k_i<\ell_i$. Now we propose the following conjecture.
       \begin{conj} For ${\bf k}\succ {\bf l}$, we have\\
              $(i)$  $\widetilde{\mathrm{Li}}^{\star}_{{\bf k}}(z) \succ  \widetilde{\mathrm{Li}}^{\star}_{{\bf l}}(z)    $ on $\mathcal{U}$. \\
      $(ii)$ Furthermore, denote by $p_{{\bf k},{\bf l}}$  the holomorphic map $p_{{\bf k},{\bf l}}: \mathcal{U}\rightarrow \mathcal{U} $ which satisfies 
       \[
        \widetilde{\mathrm{Li}}^{\star}_{{\bf k}}\left( p_{{\bf k},{\bf l}}(z)\right) = \widetilde{\mathrm{Li}}^{\star}_{{\bf l}}(z)        \]
        on $\mathcal{U}$. Let 
        \(
        p_{{\bf k},{\bf l}} (z) =\sum_{n\geq 1} c_n z^n\), then 
        \[
        c_n>0,\quad \forall\, n\geq 1,\quad \sum_{n\geq 1} c_n\leq 1.
        \]
       \end{conj}

The paper is organized as follows.  Section 2 constructs the finite and infinite-depth functions and proves the special limiting identities needed later.  Section 3 establishes complete monotonicity and univalence.  Section 4 proves strict separation of unit-circle curves by the Hausdorff--Stieltjes and monotone likelihood-ratio method.  Section 5 proves Theorem \ref{one} by introducing the binary index coordinate and applying the invariance of domain theorem to a proper injective map.  Section 6 proves the cyclotomic density and the derived-set statement.

The main results of this paper are inspired by the following beautiful formula
  \[
  \frac{1}{1-e^{i\theta}}=\frac{1}{2}+\frac{i}{2}\cdot \mathrm{cot}\,\frac{\theta}{2}, \;\theta\in\mathbb{R}.
  \]

\section{Finite and infinite multiple star polylogarithms}\label{sec:limits}
For $m,r\geq 1$, put
\[
S_r(m)=\sum_{m\geq n_1\geq\cdots\geq n_r\geq 1}
\frac{1}{n_1\cdots n_r}.
\]

\begin{lem}\label{lim}
$(i)$ If $m\geq 2$, then
\[
1<S_r(m)<m.
\]

$(ii)$ For every $m\geq 1$, the sequence $S_r(m)$ is increasing and
\[
\lim_{r\rightarrow+\infty}S_r(m)=m.
\]

$(iii)$ Let $k_1,k_2,\ldots\geq 1$.  For fixed $m\geq 1$, the limit
\[
\lim_{r\rightarrow+\infty}
\sum_{m\geq n_1\geq\cdots\geq n_r\geq 1}
\frac{1}{n_1^{k_1}\cdots n_r^{k_r}}
\]
exists.
\end{lem}

\begin{proof}
For $(i)$, the lower bound follows from the two terms with all $n_j=1$ and all $n_j=2$.  On the other hand,
\[
\begin{split}
\sum_{m\geq n_1\geq\cdots\geq n_r\geq 1}
\frac{1}{n_1\cdots n_r}
&<\sum_{m\geq n_1\geq\cdots\geq n_r\geq 1}
\frac{1}{n_1\cdots n_{r-1}}\\
&=\sum_{m\geq n_1\geq\cdots\geq n_{r-1}\geq 1}
\frac{n_{r-1}}{n_1\cdots n_{r-1}}\\
&=\cdots=m.
\end{split}
\]

For $(ii)$, separating the strings according to the last place at which the value is greater than $1$ gives
\[
S_r(m)=1+\sum_{m\geq n_1\geq 2}\frac1{n_1}+
\cdots+
\sum_{m\geq n_1\geq\cdots\geq n_r\geq 2}
\frac1{n_1\cdots n_r}.
\]
Hence $S_r(m)$ is increasing, and
\[
\begin{split}
\lim_{r\rightarrow+\infty}S_r(m)
&=\prod_{n=2}^{m}\left(1+\frac1n+\frac1{n^2}+\cdots\right)\\
&=\prod_{n=2}^{m}\frac{n}{n-1}=m.
\end{split}
\]
For $(iii)$, the corresponding sequence is increasing, because the terms with $n_{r+1}=1$ reproduce the preceding sum.  It is bounded above by $S_r(m)<m$.  Therefore the limit exists.
\end{proof}

For a nonempty finite index ${\bf k}=(k_1,\ldots,k_r)$, define
\begin{equation}\label{Adef}
A_{\bf k}(n)=\frac{1}{(n+1)^{k_1}}
\sum_{n+1\geq n_2\geq\cdots\geq n_r\geq 1}
\frac{1}{n_2^{k_2}\cdots n_r^{k_r}},\quad   n\geq 0.
\end{equation}
where the empty inner sum is $1$.  Then, for $|z|<1$,
\begin{equation}\label{Hseries}
H_{\bf k}(z):=\widetilde{\mathrm{Li}}^{\star}_{\bf k}(z)
=\sum_{n=0}^{+\infty}A_{\bf k}(n)z^n.
\end{equation}

\begin{prop}\label{equ}
For every finite index ${\bf k}=(k_1,\ldots,k_r)$,
\begin{equation}\label{cubeint}
H_{\bf k}(z)=
\int_{[0,1]^{k_1+\cdots+k_r}}
\frac{dt_1\cdots dt_{k_1+\cdots+k_r}}
{(1-zt_1\cdots t_{k_1})
 (1-zt_1\cdots t_{k_1+k_2})\cdots
 (1-zt_1\cdots t_{k_1+\cdots+k_r})}
\end{equation}
holds for $|z|<1$.  The right-hand side is holomorphic on $\Lambda$ and gives the analytic continuation of $H_{\bf k}$ to $\Lambda$.
\end{prop}

\begin{proof}
For $|z|<1$, expand the $j$-th denominator in a geometric series with exponent $a_j\geq0$ and integrate term by term.  With
\[
n_j=1+a_j+a_{j+1}+\cdots+a_r\qquad(1\leq j\leq r),
\]
one has $n_1\geq\cdots\geq n_r\geq1$, the power of $z$ is $n_1-1$, and the integrations over the $j$-th block contribute $n_j^{-k_j}$.  This gives \eqref{Hseries}.  If $K\Subset\Lambda$, then all factors $1-zt$ are bounded away from zero uniformly for $z\in K$ and $0\leq t\leq1$.  Dominated differentiation under the integral therefore proves holomorphy on $\Lambda$, and the identity theorem gives the claimed analytic continuation.
\end{proof}

Let ${\bf k}=(k_1,\cdots,k_r,\ldots)\in\widehat{\mathcal{T}}$.  By Lemma \ref{lim}, $(iii)$, the limit
\begin{equation}\label{Ainfty}
A_{\bf k}(n)=\lim_{r\rightarrow+\infty}A_{k_1,\ldots,k_r}(n)
\end{equation}
exists for every $n\geq 0$.  Notice that
\begin{equation}\label{Abound}
0<A_{\bf k}(n)\leq 1,
\qquad A_{\bf k}(0)=1.
\end{equation}
Indeed, Lemma \ref{lim}, $(ii)$, gives the upper bound after replacing all exponents by $1$.

\begin{prop}\label{exi}
For every ${\bf k}\in\widehat{\mathcal{T}}$, the sequence $H_{k_1,\ldots,k_r}$ converges locally uniformly on $\mathcal{D}$ to a holomorphic function $H_{\bf k}$.  For $|z|<1$,
\begin{equation}\label{Hinftyseries}
H_{\bf k}(z)=\sum_{n=0}^{+\infty}A_{\bf k}(n)z^n.
\end{equation}
If ${\bf k}\in\mathcal{T}$, then the boundary value
\[
H_{\bf k}(1):=\lim_{r\rightarrow+\infty}\zeta^{\star}(k_1,\ldots,k_r)
\]
exists and is greater than $1$.
\end{prop}

\begin{proof}
For $|z|<1$, \eqref{Abound} and dominated convergence give \eqref{Hinftyseries}.  We next prove local uniform boundedness on $\mathcal{D}$.  Let $\mathrm{Re}\,z\leq 1-\varepsilon$ with $\varepsilon>0$.  If $0<\varepsilon<1$, then for $0\leq t\leq1$ one has $|1-zt|\geq1-(1-\varepsilon)t$.  After replacing every original denominator by its lower bound, insert the missing intermediate-product factors.  Each inserted factor lies in $[\varepsilon,1]$, so this can only decrease the denominator and hence increase the majorant.  Thus the absolute value of \eqref{cubeint} is at most
\[
\int_{[0,1]^N}
\frac{dt_1\cdots dt_N}
{[1-(1-\varepsilon)t_1]
 [1-(1-\varepsilon)t_1t_2]\cdots
 [1-(1-\varepsilon)t_1\cdots t_N]},
\qquad N=k_1+\cdots+k_r.
\]
The displayed integral is $H_{\{1\}^{N}}(1-\varepsilon)$.  Its coefficient of $(1-\varepsilon)^n$ is $A_{\{1\}^{N}}(n)$, which is at most $1$ by Lemma \ref{lim}, $(i)$ (and equals $1$ when $n=0$).  Hence
\[
|H_{k_1,\ldots,k_r}(z)|
\leq H_{\{1\}^{N}}(1-\varepsilon)
\leq\sum_{n=0}^{+\infty}(1-\varepsilon)^n=\frac1\varepsilon.
\]
For $\varepsilon\geq 1$, the bound $1$ is immediate.  Hence the family is locally uniformly bounded on $\mathcal{D}$.  Since it converges on the open unit disc, the Vitali--Porter theorem \cite{sch} gives locally uniform convergence throughout $\mathcal{D}$.

It remains to consider $z=1$.  The sequence $\zeta^{\star}(k_1,\ldots,k_r)$ is increasing.  If $k_1\geq 3$, then
\[
\zeta^{\star}(k_1,\ldots,k_r)
\leq\sum_{m=1}^{+\infty}\frac{m}{m^{k_1}}
=\zeta(k_1-1).
\]
If $k_1=2$ and $k_s\geq 2$ for some $s\geq 2$, then, for $r\geq s$ and after summing the tail by Lemma \ref{lim},
\[
\zeta^{\star}(k_1,\ldots,k_r)
\leq\zeta^{\star}(k_1,\ldots,k_{s-1},k_s-1)<+\infty.
\]
Thus the increasing sequence is bounded and its limit exists.  It is greater than $1$ because every finite partial value is greater than $1$.
\end{proof}

The two exceptional types excluded at $z=1$ are genuine.  If $k_1=1$, the first summation already has harmonic divergence.  For the remaining exceptional index, Lemma \ref{lim}, $(ii)$ gives the self-contained estimate
\[
\zeta^{\star}(2,\{1\}^{r})
=\sum_{m=1}^{+\infty}\frac{S_r(m)}{m^2}.
\]
For every fixed $M$, monotone convergence in the finite sum yields
\[
\liminf_{r\rightarrow+\infty}\zeta^{\star}(2,\{1\}^{r})
\geq\sum_{m=1}^{M}\frac1m.
\]
Letting $M\rightarrow+\infty$ proves that the sequence attached to $(2,1,1,\ldots)$ tends to $+\infty$.

We shall use the following special limits repeatedly.  As usual, $\{1\}^{r}$ denotes a string of $r$ ones.

\begin{prop}\label{speciallimits}
All limits below are locally uniform on $\mathcal{D}$.

$(i)$
\[
\lim_{r\rightarrow+\infty}H_{\{1\}^{r}}(z)=\frac1{1-z}.
\]

$(ii)$ For a finite nonempty index $(k_1,\ldots,k_s)$,
\[
\lim_{r\rightarrow+\infty}
H_{k_1,\ldots,k_{s-1},k_s+1,\{1\}^{r}}(z)
=H_{k_1,\ldots,k_s}(z).
\]

$(iii)$ For a finite nonempty index ${\bf k}$,
\[
\lim_{q\rightarrow+\infty}H_{{\bf k},q}(z)=H_{\bf k}(z).
\]
\end{prop}

\begin{proof}
It suffices first to work in $|z|<1$ and then apply local uniform boundedness and the Vitali--Porter theorem.  For $(i)$, Lemma \ref{lim}, $(ii)$, gives
\[
\lim_{r\rightarrow+\infty}A_{\{1\}^{r}}(n)=1,
\]
so the limit is $\sum_{n\geq 0}z^n=(1-z)^{-1}$.

For $(ii)$, fix the first $s$ summation variables.  The infinite tail of ones contributes the factor
\[
\lim_{r\rightarrow+\infty}
\sum_{n_s\geq n_{s+1}\geq\cdots\geq n_{s+r}\geq 1}
\frac1{n_{s+1}\cdots n_{s+r}}=n_s
\]
by Lemma \ref{lim}, $(ii)$.  This cancels one power of $n_s$.

For $(iii)$, the difference $H_{{\bf k},q}-H_{\bf k}$ is obtained by restricting the last summation variable to values at least $2$.  For every $n$ its coefficient tends to zero as $q\rightarrow+\infty$, while for $q\geq1$ it is bounded by $A_{({\bf k},1)}(n)-A_{\bf k}(n)\leq1$.  Hence the difference series is dominated by $\sum_{n\geq0}|z|^n$ on every smaller disc.  Dominated convergence gives the result for $|z|<1$, and Vitali--Porter completes the proof.
\end{proof}

For a finite prefix ${\bf p}=(p_1,\ldots,p_r)$ and an infinite extension ${\bf k}=({\bf p},k_{r+1},\ldots)$, the next estimate controls the influence of the tail.

\begin{lem}\label{tailosc}
Let $m=n+1$.  Then
\begin{equation}\label{tailoscineq}
0\leq A_{\bf k}(n)-A_{\bf p}(n)
\leq 1-A_{\{1\}^{r}}(n).
\end{equation}
In particular, the oscillation of $A_{\bf k}(n)$ among all infinite extensions of a fixed prefix of length $r$ tends to zero as $r\rightarrow+\infty$, uniformly in the prefix.
\end{lem}

\begin{proof}
If $r=1$, Lemma \ref{lim} bounds the tail contribution between $1$ and $m$, and therefore
\[
0\leq A_{\bf k}(n)-A_{(p_1)}(n)
\leq\frac{m-1}{m^{p_1}}
\leq\frac{m-1}{m}=1-A_{(1)}(n).
\]
Assume now that $r\geq2$.  For fixed $m\geq n_2\geq\cdots\geq n_r$, the infinite tail contributes a number between $1$ and $n_r$, by Lemma \ref{lim}.  Therefore
\[
\begin{split}
0\leq A_{\bf k}(n)-A_{\bf p}(n)
&\leq\frac1m
\sum_{m\geq n_2\geq\cdots\geq n_r\geq 1}
\frac{n_r-1}{n_2\cdots n_r}.
\end{split}
\]
Now
\[
\frac1m
\sum_{m\geq n_2\geq\cdots\geq n_r\geq 1}
\frac{n_r}{n_2\cdots n_r}=1,
\]
whereas the same sum with $n_r$ replaced by $1$ is $A_{\{1\}^{r}}(n)$.  This proves \eqref{tailoscineq}.  The last assertion follows from Proposition \ref{speciallimits}, $(i)$.
\end{proof}

\section{Completely monotone sequences and univalent functions}\label{sec:cm}
We now place the functions of Section \ref{sec:limits} in the Hausdorff--Stieltjes class.  Besides proving Theorem \ref{cm}, the constructions in this section provide the representing densities used in the strict separation argument.

\begin{Def}
Let $\Delta x_n=x_n-x_{n+1}$.  A sequence $\{x_n\}_{n\geq 0}$ of nonnegative real numbers is called completely monotone if
\[
\Delta^m x_n\geq 0,
\qquad m,n\geq 0.
\]
\end{Def}

We recall the classical theorem of Hausdorff \cite{hau,wid}.

\begin{Thm}[Hausdorff]\label{hausdorff}
A sequence $\{x_n\}_{n\geq 0}$ is completely monotone if and only if there is a finite nonnegative Borel measure $\mu$ on $[0,1]$ such that
\[
x_n=\int_{[0,1]}t^n\,d\mu(t),
\qquad n\geq 0.
\]
If $x_0=1$, then $\mu$ is a probability measure.
\end{Thm}

For a nonnegative integrable function $u$ on $(0,1)$, define
\begin{equation}\label{MCdef}
(\mathsf{M}u)(t)=\int_t^1\frac{u(v)}{v}\,dv,
\qquad
(\mathsf{C}u)(t)=\int_0^t\frac{u(v)}{1-v}\,dv.
\end{equation}
Tonelli's theorem gives, for every $n\geq 0$,
\begin{align}
\int_0^1t^n(\mathsf{M}u)(t)\,dt
&=\frac1{n+1}\int_0^1t^nu(t)\,dt,\label{Mmoment}\\
\int_0^1t^n(\mathsf{C}u)(t)\,dt
&=\frac1{n+1}\sum_{j=0}^{n}\int_0^1t^ju(t)\,dt.\label{Cmoment}
\end{align}
Taking $n=0$ shows that both operators preserve the total integral.

For a nonempty finite index ${\bf k}=(k_1,\ldots,k_r)$, put
\begin{equation}\label{sigmadef}
\sigma_{\bf k}
=\mathsf{M}^{k_1-1}\mathsf{C}\mathsf{M}^{k_2-1}\mathsf{C}\cdots
\mathsf{C}\mathsf{M}^{k_r-1}{\bf 1},
\end{equation}
where no $\mathsf{C}$ occurs when $r=1$, and composition is read from right to left. Here ${\bf 1}$ means the constant function $1$ on $[0,1]$. Every such function is positive and locally absolutely continuous on $(0,1)$.

\begin{prop}\label{momentdensity}
For every nonempty finite index ${\bf k}$ and every $n\geq 0$,
\begin{equation}\label{momentrepresentation}
A_{\bf k}(n)=\int_0^1t^n\sigma_{\bf k}(t)\,dt.
\end{equation}
Consequently, $\sigma_{\bf k}(t)dt$ is a probability measure, the sequence $\{A_{\bf k}(n)\}_{n\geq 0}$ is completely monotone, and
\begin{equation}\label{StieltjesH}
H_{\bf k}(z)=\int_0^1\frac{\sigma_{\bf k}(t)}{1-tz}\,dt,
\qquad z\in\Lambda.
\end{equation}
\end{prop}

\begin{proof}
If $p\geq 2$, then directly from \eqref{Adef},
\begin{equation}\label{ArecM}
A_{(p,{\boldsymbol\tau})}(n)
=\frac1{n+1}A_{(p-1,{\boldsymbol\tau})}(n).
\end{equation}
If ${\boldsymbol\tau}$ is nonempty, changing the outer variable in the inner sum gives
\begin{equation}\label{ArecC}
A_{(1,{\boldsymbol\tau})}(n)
=\frac1{n+1}\sum_{j=0}^{n}A_{\boldsymbol\tau}(j).
\end{equation}
Starting with $A_{(1)}(n)=1/(n+1)=\int_0^1t^n dt$, the assertion follows inductively from \eqref{Mmoment} and \eqref{Cmoment}.  Since $A_{\bf k}(0)=1$, the density has total mass $1$.  The Hausdorff theorem proves complete monotonicity, and summing the moments for $|z|<1$, followed by analytic continuation, gives \eqref{StieltjesH}.
\end{proof}

\begin{Cor}\label{boundaryvalues}
For every nonempty finite index ${\bf k}$, the sequence $A_{\bf k}(n)$ decreases to $0$.  If $|z|=1$ and $z\neq1$, then the series in \eqref{Hseries} converges and its sum equals the boundary value of either integral representation \eqref{cubeint} or \eqref{StieltjesH}.
\end{Cor}

\begin{proof}
The moment representation makes $A_{\bf k}(n)$ nonincreasing, and dominated convergence gives
\[
A_{\bf k}(n)=\int_0^1t^n\sigma_{\bf k}(t)\,dt\rightarrow0,
\]
because $\sigma_{\bf k}(t)dt$ is absolutely continuous and hence has no atom at $1$.  The Dirichlet test now proves convergence on $S^1\setminus\{1\}$.  Abel's theorem identifies the sum with the radial limit of $H_{\bf k}$.  In \eqref{StieltjesH}, the denominator is bounded away from zero for a fixed $z\in S^1\setminus\{1\}$, so dominated convergence identifies that radial limit with the Stieltjes integral and hence, by Proposition \ref{equ}, with the cube integral.
\end{proof}

We shall use two standard results from geometric function theory.  The first is due to Wirths \cite{wir}.

\begin{Thm}[Wirths]\label{wirth}
If $\{x_n\}_{n\geq 0}$ is completely monotone and $x_0=1$, then
\[
f(z)=\sum_{n=0}^{+\infty}x_nz^{n+1}
\]
extends holomorphically to $\Lambda$ and is univalent on $\mathcal{D}$.
\end{Thm}

\begin{lem}\label{closeconvex}
Let $G$ be a domain and let $f,g$ be holomorphic on $G$.  Suppose that $g$ is univalent, $g(G)$ is convex, and
\[
\mathrm{Re}\,\frac{f'(z)}{g'(z)}>0,
\qquad z\in G.
\]
Then $f$ is univalent on $G$.
\end{lem}

This is the usual close-to-convexity criterion; see Section 2.4.1 in  \cite{GK}.  It follows by applying the Noshiro--Warschawski theorem to $f\circ g^{-1}$ on the convex domain $g(G)$.

\begin{lem}\label{shiftedunivalent}
Let $\mu$ be a probability measure on $[0,1]$ with
\[
\int_{[0,1]}t\,d\mu(t)>0,
\]
and set
\[
G(z)=\int_{[0,1]}\frac{d\mu(t)}{1-tz}.
\]
Then $G$ is univalent on $\mathcal{D}$.
\end{lem}

\begin{proof}
The function
\[
g(z)=-\log(1-z)
\]
is univalent on $\mathcal{D}$ and maps $\mathcal{D}$ onto the horizontal strip $|\mathrm{Im}\,w|<\pi/2$, which is convex.  Since $g'(z)=(1-z)^{-1}$, it is enough by Lemma \ref{closeconvex} to prove
\[
\mathrm{Re}\big((1-z)G'(z)\big)>0.
\]
Write $z=a+ib$, where $a<1$.  For $0<t\leq 1$, a direct calculation gives
\begin{equation}\label{positivekernel}
\mathrm{Re}\left(\frac{t(1-z)}{(1-tz)^2}\right)
=
\frac{t\big((1-a)(1-ta)^2+t b^2[2-(1+a)t]\big)}
{\big((1-ta)^2+t^2b^2\big)^2}>0.
\end{equation}
Indeed, $1-a>0$, and $2-(1+a)t>0$: if $a\geq-1$ its minimum on $[0,1]$ is $1-a$, while if $a<-1$ its minimum is $2$.  Integrating \eqref{positivekernel} against $\mu$ proves the assertion, because $\mu((0,1])>0$.
\end{proof}

\begin{proof}[Proof of Theorem \ref{cm}]
Part $(i)$ is Proposition \ref{momentdensity}.  Part $(ii)$ follows from Theorem \ref{wirth}, applied to $x_n=A_{\bf k}(n)$, because
\[
\mathrm{Li}^{\star}_{\bf k}(z)=zH_{\bf k}(z)
=\sum_{n=0}^{+\infty}A_{\bf k}(n)z^{n+1}.
\]
Part $(iii)$ follows from Lemma \ref{shiftedunivalent} and \eqref{StieltjesH}.

For an infinite index ${\bf k}$, every finite difference $\Delta^m A_{k_1,\ldots,k_r}(n)$ is nonnegative.  Passing to the limit in $r$ shows that $\{A_{\bf k}(n)\}_{n\geq 0}$ is completely monotone.  By Theorem \ref{hausdorff}, there is a probability measure $\mu_{\bf k}$ on $[0,1]$ such that
\begin{equation}\label{infiniteStieltjes}
A_{\bf k}(n)=\int_{[0,1]}t^n\,d\mu_{\bf k}(t),
\qquad
H_{\bf k}(z)=\int_{[0,1]}\frac{d\mu_{\bf k}(t)}{1-tz}.
\end{equation}
The second equality first holds on the unit disc and then on $\Lambda$ by the identity theorem.  Moreover, the term with outer summation variable $2$ and all later variables equal to $1$ gives
\[
A_{\bf k}(1)\geq 2^{-k_1}>0.
\]
Hence Lemma \ref{shiftedunivalent} applies to $H_{\bf k}$, and Theorem \ref{wirth} applies to $zH_{\bf k}$.  The locally uniform convergence asserted in $(iv)$ was proved in Proposition \ref{exi}.  This completes the proof.
\end{proof}

\section{Strict separation of the unit circle curves}\label{sec:separation}
For the rest of the paper it is convenient to write
\[
F_{\bf k}(z)=H_{\bf k}(z)-1.
\]
For finite indices, this is the series beginning with the first nonconstant coefficient,
\[
F_{\bf k}(z)=\sum_{n=1}^{+\infty}A_{\bf k}(n)z^n.
\]
We first establish the analytic closure properties needed to compare two such functions.

Let $\mathfrak{T}$ denote the class of all functions
\begin{equation}\label{Tclass}
h(z)=\int_{[0,1]}\frac{d\mu(t)}{1-tz},
\end{equation}
where $\mu$ is a probability measure on $[0,1]$.  Thus $h(0)=1$.  Also,
\begin{equation}\label{atomzero}
\lim_{x\rightarrow+\infty}h(-x)=\mu(\{0\}).
\end{equation}
The class $\mathfrak{T}$ is closed under locally uniform convergence on $\Lambda$.  Indeed, the representing probability measures are weakly compact, and every weakly convergent subsequence gives the same analytic limit in \eqref{Tclass}.

\begin{lem}\label{Stieltjescriterion}
Suppose that $h$ is holomorphic on $\Lambda$, $h(0)=1$, $h(x)>0$ for $x<1$, and
\[
\mathrm{Im}\,h(z)\geq 0
\]
whenever $\mathrm{Im}\,z>0$.  If
\begin{equation}\label{growthcriterion}
\frac{h(iy)}{iy}\rightarrow 0
\qquad (y\rightarrow+\infty),
\end{equation}
then $h\in\mathfrak{T}$.
\end{lem}

\begin{proof}
By the Herglotz--Nevanlinna representation theorem (Theorem I in Chapter $2$, \cite{Donoghue1974}),
\[
h(z)=a+bz+\int_{\mathbb{R}}
\left(\frac1{u-z}-\frac{u}{1+u^2}\right)d\rho(u),
\qquad a\in\mathbb{R},\quad b\geq 0,
\]
where $\rho$ is a positive measure.  Since $h$ is holomorphic across $(-\infty,1)$ and is real there, the support criterion for the Herglotz measure gives
\[
\mathrm{supp}\,\rho\subseteq[1,+\infty).
\]
Moreover, $b=\lim_{y\rightarrow+\infty}h(iy)/(iy)$, so \eqref{growthcriterion} implies $b=0$.

For $x<1$,
\[
h'(x)=\int_{[1,+\infty)}\frac{d\rho(u)}{(u-x)^2}\geq 0.
\]
Thus $h$ is nondecreasing on $(-\infty,1)$, and the limit
\[
c=\lim_{x\rightarrow+\infty}h(-x)
\]
exists in $[0,1]$.  Integrating $h'$ from $-R$ to $0$ and using Tonelli's theorem gives
\[
1-h(-R)=\int_{[1,+\infty)}
\left(\frac1u-\frac1{u+R}\right)d\rho(u).
\]
Monotone convergence as $R\rightarrow+\infty$ therefore yields
\[
\int_{[1,+\infty)}\frac{d\rho(u)}u=1-c<+\infty.
\]
For fixed $z\in\Lambda$, the difference kernel in $h(z)-h(-R)$ is dominated by a constant multiple of $1/u$ on $[1,+\infty)$.  Dominated convergence consequently gives
\[
h(z)=c+\int_{[1,+\infty)}\frac{d\rho(u)}{u-z}
=c+\int_{[1,+\infty)}\frac{u}{u-z}\,d\lambda(u),
\qquad d\lambda(u)=\frac{d\rho(u)}u.
\]
Since $c+\lambda([1,+\infty))=1$, the change of variables $t=1/u$, with the constant $c$ regarded as an atom at $0$, yields \eqref{Tclass}.
\end{proof}

\begin{lem}\label{quotientlemma}
Suppose that
\[
f(z)=\int_0^1\frac{\phi(t)dt}{1-tz},
\qquad
g(z)=c+\int_0^1\frac{\psi(t)dt}{1-tz},
\]
where $c\geq 0$, $f,g\in\mathfrak{T}$, $\phi,
\psi>0$ on $(0,1)$, and $\phi/\psi$ is nondecreasing.  If $c=0$, assume in addition that
\begin{equation}\label{psiintegrable}
\int_0^1\frac{\psi(t)}t\,dt<+\infty.
\end{equation}
Then $f/g\in\mathfrak{T}$.
\end{lem}

\begin{proof}
The function $g$ has no zeros in $\Lambda$: its imaginary part is strictly positive in the upper half-plane and strictly negative in the lower half-plane, while $g(x)>0$ for $x<1$.  Hence $f/g$ is holomorphic on $\Lambda$.  For $\mathrm{Im}\,z>0$, expansion of
$f(z)\overline{g(z)}-\overline{f(z)}g(z)$ gives
\begin{align}
2i|g(z)|^2\mathrm{Im}\frac{f(z)}{g(z)}
&=(z-\overline z)
\int_{0<s<t<1}
\frac{(t-s)(\phi(t)\psi(s)-\phi(s)\psi(t))}
{|1-sz|^2|1-tz|^2}\,ds\,dt\notag\\
&\quad +(z-\overline z)c
\int_0^1\frac{t\phi(t)}{|1-tz|^2}\,dt.\label{pickquotient}
\end{align}
The right-hand side has the required sign because $\phi/\psi$ is nondecreasing.  Hence $\mathrm{Im}(f/g)(z)\geq 0$ in the upper half-plane.  Also, $(f/g)(x)>0$ for $x<1$ and $(f/g)(0)=1$.

If $c=0$, then \eqref{psiintegrable} and dominated convergence give
\[
iy\,g(iy)\longrightarrow-\int_0^1\frac{\psi(t)}t\,dt\neq 0,
\qquad f(iy)\longrightarrow 0.
\]
Consequently $(f/g)(iy)/(iy)\rightarrow 0$.  If $c>0$, then $g(iy)\rightarrow c$ and $|f(iy)|\leq 1$, so the same growth condition holds.  Lemma \ref{Stieltjescriterion} completes the proof.
\end{proof}

\begin{lem}\label{involution}
If $h\in\mathfrak{T}$, then
\begin{equation}\label{Jdef}
(\mathcal{J}h)(z)=\frac1{(1-z)h(z)}
\end{equation}
belongs to $\mathfrak{T}$.
\end{lem}

\begin{proof}
Let $\mu$ represent $h$.  If $\mathrm{Im}\,z>0$, then
\[
(1-z)h(z)=\int_{[0,1]}\frac{1-z}{1-tz}\,d\mu(t),
\qquad
\mathrm{Im}\frac{1-z}{1-tz}
=-\frac{(1-t)\mathrm{Im}\,z}{|1-tz|^2}\leq 0.
\]
Unless $\mu=\delta_1$, the imaginary part is strictly negative; for $\mu=\delta_1$, the product is identically $1$.  Thus its reciprocal maps the upper half-plane into the closed upper half-plane, is positive on $(-\infty,1)$, and equals $1$ at the origin.  Finally,
\[
\mathrm{Re}\frac{1-iy}{1-ity}
=\frac{1+ty^2}{1+t^2y^2}\geq 1,
\]
so $|(1-iy)h(iy)|\geq 1$ and $(\mathcal{J}h)(iy)/(iy)\rightarrow 0$.  Lemma \ref{Stieltjescriterion} applies.
\end{proof}

For two nonempty finite indices
\[
{\bf k}=(k_1,\ldots,k_r),
\qquad
{\boldsymbol\ell}=(\ell_1,\ldots,\ell_s),
\]
write ${\bf k}\succ{\boldsymbol\ell}$ if either ${\boldsymbol\ell}$ is a proper prefix of ${\bf k}$, or, at the first position where the two indices differ, the component of ${\bf k}$ is smaller.  Thus a proper extension is larger, and the order is reverse lexicographic at the first unequal component.

\begin{prop}\label{strictMLR}
If ${\bf k}\succ{\boldsymbol\ell}$, then
\begin{equation}\label{MLRratio}
t\mapsto\frac{\sigma_{\bf k}(t)}{\sigma_{\boldsymbol\ell}(t)}
\end{equation}
is strictly increasing on $(0,1)$.
\end{prop}

\begin{proof}
We first record a preservation property.  Let $u,v>0$ be locally absolutely continuous, suppose that the transforms below are finite on $(0,1)$, and assume that $u/v$ is nondecreasing.  Then
\[
\frac{\mathsf{C}u(t)}{\mathsf{C}v(t)}
=
\frac{\displaystyle\int_0^t\frac{u(s)}{v(s)}\frac{v(s)}{1-s}\,ds}
{\displaystyle\int_0^t\frac{v(s)}{1-s}\,ds},
\]
\[
\frac{\mathsf{M}u(t)}{\mathsf{M}v(t)}
=
\frac{\displaystyle\int_t^1\frac{u(s)}{v(s)}\frac{v(s)}s\,ds}
{\displaystyle\int_t^1\frac{v(s)}s\,ds}.
\]
The first quotient is a weighted average of $u/v$ over $(0,t)$ and the second is a weighted average over $(t,1)$.  Differentiating almost everywhere gives
\[
\left(\frac{\mathsf{C}u}{\mathsf{C}v}\right)'
=
\frac{v(t)}{(1-t)(\mathsf{C}v)(t)}
\left(\frac{u(t)}{v(t)}-\frac{(\mathsf{C}u)(t)}{(\mathsf{C}v)(t)}\right)\geq 0,
\]
\[
\left(\frac{\mathsf{M}u}{\mathsf{M}v}\right)'
=
\frac{v(t)}{t(\mathsf{M}v)(t)}
\left(\frac{(\mathsf{M}u)(t)}{(\mathsf{M}v)(t)}-\frac{u(t)}{v(t)}\right)\geq 0.
\]
If $u/v$ is strictly increasing, positivity of the weights shows that the first weighted average is strictly smaller than the right-endpoint value and that the second is strictly larger than the left-endpoint value.  The two displayed derivatives are therefore positive almost everywhere.  Since the quotients are locally absolutely continuous and the displayed derivatives are positive almost everywhere on every compact subinterval, both $\mathsf{C}$ and $\mathsf{M}$ preserve strictly increasing ratios.

Suppose first that, after deleting a common prefix, the two indices begin with $(p,{\boldsymbol\alpha})$ and $(q,{\boldsymbol\gamma})$, where $p<q$.  Put
\[
V_{\boldsymbol\alpha}=
\begin{cases}
{\bf 1},&{\boldsymbol\alpha}=\varnothing,\\
\mathsf{C}\sigma_{\boldsymbol\alpha},&{\boldsymbol\alpha}\neq\varnothing.
\end{cases}
\]
Then
\[
\sigma_{(p,{\boldsymbol\alpha})}=\mathsf{M}^{p-1}V_{\boldsymbol\alpha},
\qquad
\sigma_{(q,{\boldsymbol\gamma})}
=\mathsf{M}^{p-1}\big(\mathsf{M}^{q-p}V_{\boldsymbol\gamma}\big).
\]
The numerator $V_{\boldsymbol\alpha}$ is nondecreasing.  Since $q-p\geq1$ and every function involved is positive,
\[
(\mathsf{M}u)'(t)=-\frac{u(t)}t<0,
\]
so $\mathsf{M}^{q-p}V_{\boldsymbol\gamma}$ is strictly decreasing.  Their ratio is therefore strictly increasing.  Apply $\mathsf{M}^{p-1}$ simultaneously to numerator and denominator, and then restore the common prefix.  Restoring a component $a$ means applying $\mathsf{C}$ and then $\mathsf{M}^{a-1}$ simultaneously, so strict increase is preserved.

If ${\boldsymbol\ell}$ is a proper prefix of ${\bf k}$, delete the common components preceding the final component $p$ of the shorter index.  It remains to compare
\[
\frac{\mathsf{M}^{p-1}(\mathsf{C}\sigma_{\boldsymbol\alpha})}
{\mathsf{M}^{p-1}{\bf 1}},
\qquad {\boldsymbol\alpha}\neq\varnothing.
\]
The initial ratio $(\mathsf{C}\sigma_{\boldsymbol\alpha})/{\bf 1}$ is strictly increasing, and the same preservation argument completes the proof.
\end{proof}

\begin{Cor}\label{firstmomentstrict}
If ${\bf k}\succ{\boldsymbol\ell}$, then
\[
A_{\bf k}(1)>A_{\boldsymbol\ell}(1).
\]
\end{Cor}

\begin{proof}
Use $\sigma_{\boldsymbol\ell}(t)dt$ as a probability measure and put
$r(t)=\sigma_{\bf k}(t)/\sigma_{\boldsymbol\ell}(t)$.  Since $\mathbb{E}r=1$,
\[
\begin{split}
A_{\bf k}(1)-A_{\boldsymbol\ell}(1)
&=\mathbb{E}(tr)-\mathbb{E}(t)\mathbb{E}(r)\\
&=\frac12\int_0^1\int_0^1
(t-s)(r(t)-r(s))
\sigma_{\boldsymbol\ell}(t)\sigma_{\boldsymbol\ell}(s)\,dt\,ds>0.
\end{split}
\]
The strict inequality follows from Proposition \ref{strictMLR}.
\end{proof}

By \eqref{StieltjesH},
\begin{equation}\label{Fint}
F_{\bf k}(z)
=z\int_0^1\frac{t\sigma_{\bf k}(t)}{1-tz}\,dt.
\end{equation}
For $|z|=1$, $z\neq 1$, the right-hand side is also the sum of the original boundary series.  Indeed, its $N$-th partial sum is
\[
\int_0^1
\frac{tz(1-(tz)^N)}{1-tz}\sigma_{\bf k}(t)\,dt,
\]
and dominated convergence applies.

Put
\begin{equation}\label{Fhat}
\widehat F_{\bf k}(z)
=\frac{F_{\bf k}(z)}{A_{\bf k}(1)z}
=\int_0^1\frac{t\sigma_{\bf k}(t)/A_{\bf k}(1)}{1-tz}\,dt.
\end{equation}
Then $\widehat F_{\bf k}\in\mathfrak{T}$, its representing measure has no atom at $0$, and
\begin{equation}\label{denintegrable}
\int_0^1\frac{t\sigma_{\bf k}(t)/A_{\bf k}(1)}t\,dt
=\frac1{A_{\bf k}(1)}<+\infty.
\end{equation}
The positivity of the density implies that $\mathrm{Im}\,\widehat F_{\bf k}(z)$ has the sign of $\mathrm{Im}\,z$, while $\widehat F_{\bf k}(x)>0$ for $x<1$.  Thus $\widehat F_{\bf k}$ has no zeros in $\Lambda$, and $F_{\bf k}$ has exactly one zero there, the simple zero at $0$.

We use the convention $F_{\varnothing}=0$.  From \eqref{Fint},
\begin{equation}\label{Fminusinfty}
F_{\bf k}(-x)
=-\int_0^1\frac{xt}{1+xt}\sigma_{\bf k}(t)\,dt
\rightarrow -1
\qquad(x\rightarrow+\infty).
\end{equation}

\begin{prop}\label{logderivative}
For every nonempty finite index ${\bf k}$, the function
\begin{equation}\label{Pdef}
P_{\bf k}(z)=\frac{zF_{\bf k}'(z)}{F_{\bf k}(z)}
\end{equation}
has a removable singularity at $0$, belongs to $\mathfrak{T}$ after that singularity is filled in, and is not constant.
\end{prop}

\begin{proof}
Write ${\bf k}=(p,{\boldsymbol\tau})$. Assume first that $p\geq 2$, and put ${\bf k}^{-}=(p-1,{\boldsymbol\tau})$.  Then ${\bf k}^{-}\succ{\bf k}$.  Proposition \ref{strictMLR}, Lemma \ref{quotientlemma}, and \eqref{denintegrable} give
\begin{equation}\label{qminus}
q(z)=\frac{\widehat F_{{\bf k}^{-}}(z)}{\widehat F_{\bf k}(z)}
=\frac12\frac{F_{{\bf k}^{-}}(z)}{F_{\bf k}(z)}
\in\mathfrak{T},
\end{equation}
because \eqref{ArecM} gives
$A_{{\bf k}^{-}}(1)=2A_{\bf k}(1)$.  By \eqref{Fminusinfty}, $q(-x)\rightarrow1/2$, so the probability measure representing $q$ has an atom of mass $1/2$ at $0$.  The coefficient identity \eqref{ArecM} is equivalent to
\[
F_{{\bf k}^{-}}=F_{\bf k}+zF_{\bf k}'.
\]
Consequently,
\[
P_{\bf k}=\frac{F_{{\bf k}^{-}}}{F_{\bf k}}-1
=2q-1
=2\int_{(0,1]}\frac{d\nu(t)}{1-tz}\in\mathfrak{T},
\]
because the measure on the right has total mass $1$.

Now assume that $p=1$.  Put
\[
G_{\boldsymbol\tau}(z)=z+F_{\boldsymbol\tau}(z),
\qquad
h(z)=2\frac{F_{(1,{\boldsymbol\tau})}(z)}{G_{\boldsymbol\tau}(z)}.
\]
When ${\boldsymbol\tau}=\varnothing$, one has
$h=2F_{(1)}/z=\widehat F_{(1)}\in\mathfrak{T}$.  Suppose that ${\boldsymbol\tau}\neq\varnothing$ and set
\[
S=1+A_{\boldsymbol\tau}(1).
\]
The case $n=1$ of \eqref{ArecC} gives
\begin{equation}\label{halfcoefficient}
2A_{(1,{\boldsymbol\tau})}(1)=S.
\end{equation}
Therefore
\[
f(z)=\frac{2F_{(1,{\boldsymbol\tau})}(z)}{Sz}
=\int_0^1\frac{2t\sigma_{(1,{\boldsymbol\tau})}(t)/S}{1-tz}\,dt
\in\mathfrak{T},
\]
\[
g(z)=\frac{G_{\boldsymbol\tau}(z)}{Sz}
=\frac1S+\int_0^1\frac{t\sigma_{\boldsymbol\tau}(t)/S}{1-tz}\,dt
\in\mathfrak{T}.
\]
The representing measure of $g$ has a positive atom at $0$.  Moreover,
$(1,{\boldsymbol\tau})\succ{\boldsymbol\tau}$.  To see this, if all components of ${\boldsymbol\tau}$ are $1$, then ${\boldsymbol\tau}$ is a proper prefix of $(1,{\boldsymbol\tau})$; otherwise, at the first component of ${\boldsymbol\tau}$ greater than $1$, the corresponding component of $(1,{\boldsymbol\tau})$ is $1$.  Proposition \ref{strictMLR} and Lemma \ref{quotientlemma} now give
\[
h=\frac fg\in\mathfrak{T}.
\]

By Lemma \ref{involution}, $\mathcal{J}h\in\mathfrak{T}$.  Equation \eqref{Fminusinfty} gives
\[
h(-x)\sim\frac2x,
\qquad
(\mathcal{J}h)(-x)=\frac1{(1+x)h(-x)}\rightarrow\frac12.
\]
Thus the representing measure of $\mathcal{J}h$ has an atom of mass $1/2$ at $0$.

For ${\boldsymbol\tau}\neq\varnothing$, \eqref{ArecC} gives
\[
(n+1)A_{(1,{\boldsymbol\tau})}(n)
=1+\sum_{j=1}^{n}A_{\boldsymbol\tau}(j),
\]
and the same formula, with the sum absent, holds for
${\boldsymbol\tau}=\varnothing$.  Hence
\begin{equation}\label{firstoneidentity}
F_{(1,{\boldsymbol\tau})}+zF_{(1,{\boldsymbol\tau})}'
=\frac{z+F_{\boldsymbol\tau}}{1-z}
=\frac{G_{\boldsymbol\tau}}{1-z}.
\end{equation}
It follows that
\[
P_{(1,{\boldsymbol\tau})}+1
=\frac{G_{\boldsymbol\tau}}{(1-z)F_{(1,{\boldsymbol\tau})}}
=2\mathcal{J}h.
\]
Removing the atom of mass $1/2$ at $0$ shows that
$P_{(1,{\boldsymbol\tau})}=2\mathcal{J}h-1\in\mathfrak{T}$.

Finally, if
\[
F_{\bf k}(z)=A_{\bf k}(1)z+A_{\bf k}(2)z^2+\cdots,
\]
then
\[
P_{\bf k}(z)=1+\frac{A_{\bf k}(2)}{A_{\bf k}(1)}z+O(z^2).
\]
Since $A_{\bf k}(2)>0$, the function is not constant.
\end{proof}

Let $z=e^{i\theta}$, where $0<\theta<\pi$.  Equation \eqref{Fint} gives
\begin{equation}\label{ImFpositive}
\mathrm{Im}\,F_{\bf k}(e^{i\theta})
=\sin\theta\int_0^1
\frac{t\sigma_{\bf k}(t)}{|1-te^{i\theta}|^2}\,dt>0.
\end{equation}
We may therefore define
\[
\alpha_{\bf k}(\theta)=\arg F_{\bf k}(e^{i\theta})\in(0,\pi),
\qquad
\rho_{\bf k}(\theta)=|F_{\bf k}(e^{i\theta})|.
\]
Let $\eta_{\bf k}$ represent $P_{\bf k}$.  Since $P_{\bf k}$ is not constant,
$\eta_{\bf k}((0,1])>0$, and
\[
\mathrm{Re}\,P_{\bf k}(e^{i\theta})
=\int_{[0,1]}
\frac{1-t\cos\theta}{|1-te^{i\theta}|^2}\,d\eta_{\bf k}(t)>0,
\]
\[
\mathrm{Im}\,P_{\bf k}(e^{i\theta})
=\sin\theta\int_{(0,1]}
\frac{t}{|1-te^{i\theta}|^2}\,d\eta_{\bf k}(t)>0.
\]
Along the upper semicircle,
\[
\frac{d}{d\theta}\log F_{\bf k}(e^{i\theta})
=iP_{\bf k}(e^{i\theta}),
\]
and hence
\begin{equation}\label{selfmonotonicity}
\alpha_{\bf k}'(\theta)=\mathrm{Re}\,P_{\bf k}(e^{i\theta})>0,
\qquad
\frac{\rho_{\bf k}'(\theta)}{\rho_{\bf k}(\theta)}
=-\mathrm{Im}\,P_{\bf k}(e^{i\theta})<0.
\end{equation}
Thus the argument strictly increases and the modulus strictly decreases.

\begin{prop}\label{finitecross}
If ${\bf k}\succ{\boldsymbol\ell}$, then, for every $0<\theta<\pi$,
\[
|F_{\bf k}(e^{i\theta})|>|F_{\boldsymbol\ell}(e^{i\theta})|,
\qquad
\arg F_{\bf k}(e^{i\theta})>
\arg F_{\boldsymbol\ell}(e^{i\theta}).
\]
\end{prop}

\begin{proof}
Put
\[
c=\frac{A_{\bf k}(1)}{A_{\boldsymbol\ell}(1)}>1.
\]
By Proposition \ref{strictMLR}, Lemma \ref{quotientlemma}, and \eqref{denintegrable},
\begin{equation}\label{qfinite}
q_{{\bf k},{\boldsymbol\ell}}(z)
=\frac{\widehat F_{\bf k}(z)}{\widehat F_{\boldsymbol\ell}(z)}
=\frac1c\frac{F_{\bf k}(z)}{F_{\boldsymbol\ell}(z)}
\in\mathfrak{T}.
\end{equation}
By \eqref{Fminusinfty}, the representing probability measure $\nu$ of $q_{{\bf k},{\boldsymbol\ell}}$ has an atom of mass $1/c$ at $0$.  Therefore
\begin{equation}\label{strongquotientfinite}
\frac{F_{\bf k}(z)}{F_{\boldsymbol\ell}(z)}
=1+c\int_{(0,1]}\frac{d\nu(t)}{1-tz},
\qquad
\nu((0,1])=1-\frac1c>0.
\end{equation}
For $z=e^{i\theta}$, $0<\theta<\pi$, both the real and imaginary parts of $(1-tz)^{-1}$ are positive for $t>0$.  Hence the quotient in \eqref{strongquotientfinite} has real part greater than $1$ and positive imaginary part.  Its modulus is greater than $1$, and its principal argument lies in $(0,\pi/2)$.  Moreover, both $F_{\bf k}(z)$ and $F_{\boldsymbol\ell}(z)$ lie in the open upper half-plane by \eqref{ImFpositive}; their argument difference belongs to $(-\pi,\pi)$ and is therefore exactly the principal argument of the quotient.  The two strict inequalities follow.
\end{proof}

\begin{Thm}\label{finiteseparation}
If ${\bf k}\succ{\boldsymbol\ell}$, then
\[
F_{\bf k}(z_1)\neq F_{\boldsymbol\ell}(z_2)
\]
for all $|z_1|=|z_2|=1$, $z_1,z_2\neq 1$.
\end{Thm}

\begin{proof}
First let $z_j=e^{i\theta_j}$ with $0<\theta_j<\pi$.  If the two values were equal and $\theta_2\leq\theta_1$, then \eqref{selfmonotonicity} and Proposition \ref{finitecross} would give
\[
\alpha_{\boldsymbol\ell}(\theta_2)
\leq\alpha_{\boldsymbol\ell}(\theta_1)
<\alpha_{\bf k}(\theta_1),
\]
a contradiction.  Thus $\theta_2>\theta_1$, but then
\[
\rho_{\boldsymbol\ell}(\theta_2)
<\rho_{\boldsymbol\ell}(\theta_1)
<\rho_{\bf k}(\theta_1),
\]
which is again impossible.

All coefficients are real, so the same conclusion holds on the lower semicircle by conjugation.  Values on opposite open semicircles have imaginary parts of opposite signs.  At $z=-1$, \eqref{strongquotientfinite} gives
\[
\frac{F_{\bf k}(-1)}{F_{\boldsymbol\ell}(-1)}>1,
\]
while both values are negative.  Hence they are unequal there as well, and comparison with nonreal points follows from the imaginary part.
\end{proof}

We next pass from finite to infinite indices.  For an infinite index
${\bf k}=(k_1,k_2,\ldots)$, write
\[
s_j({\bf k})=k_1+\cdots+k_j.
\]
If ${\bf k}$ and ${\boldsymbol\ell}$ are distinct infinite indices, write
${\bf k}\succ{\boldsymbol\ell}$ when, at their first unequal component, the component of ${\bf k}$ is smaller.

\begin{lem}\label{betamoment}
For a finite index ${\bf p}=(p_1,\ldots,p_r)$,
\begin{equation}\label{finitebeta}
A_{\bf p}(1)=\sum_{j=1}^{r}2^{-s_j({\bf p})}.
\end{equation}
For an infinite index ${\bf k}$,
\begin{equation}\label{infinitebeta}
A_{\bf k}(1)=\beta({\bf k})
=\sum_{j=1}^{+\infty}2^{-s_j({\bf k})}.
\end{equation}
Moreover, if ${\bf k}\succ{\boldsymbol\ell}$, then
\[
\beta({\bf k})>\beta({\boldsymbol\ell}).
\]
\end{lem}

\begin{proof}
When the first summation variable equals $2$, every later variable is either $1$ or $2$, and the entries equal to $2$ form an initial block.  Formula \eqref{finitebeta} follows by summing over the length of this block.  Passing to the limit gives \eqref{infinitebeta}. The series in \eqref{infinitebeta} is the binary number whose digits equal $1$ precisely at the positions $s_j({\bf k})$.  At the first unequal component, the next digit $1$ for ${\bf k}$ occurs earlier than the next digit $1$ for ${\boldsymbol\ell}$.  The usual lexicographic comparison of nonterminating binary expansions therefore gives the last assertion.
\end{proof}

\begin{lem}\label{weakconvergence}
Let ${\bf k}^{(r)}=(k_1,\ldots,k_r)$.  Then the probability measures
$\sigma_{{\bf k}^{(r)}}(t)dt$ converge weakly to the measure $\mu_{\bf k}$ in \eqref{infiniteStieltjes}.  Consequently,
\[
H_{{\bf k}^{(r)}}\rightarrow H_{\bf k}
\]
locally uniformly on $\Lambda$, and the same is true after differentiating any fixed number of times.
\end{lem}

\begin{proof}
For each $n\geq 0$, the $n$-th moments converge by definition to $A_{\bf k}(n)$.  Every weakly convergent subsequence therefore has the same moments as $\mu_{\bf k}$.  Since polynomials are dense in $C[0,1]$, a probability measure on $[0,1]$ is determined by its moments.  Thus the whole sequence converges weakly.  For each compact $K\Subset\Lambda$, the kernels $(1-tz)^{-1}$ and all fixed $z$-derivatives form equicontinuous, uniformly bounded families in $t\in[0,1]$, uniformly for $z\in K$.  Weak convergence, followed by a finite-net argument in $z$, therefore gives the asserted locally uniform convergence, including after any fixed number of differentiations.
\end{proof}

\begin{lem}\label{noatomzero}
For every infinite index ${\bf k}$,
\begin{equation}\label{Hminusinfty}
H_{\bf k}(-x)\rightarrow 0,
\qquad
F_{\bf k}(-x)\rightarrow-1
\qquad(x\rightarrow+\infty).
\end{equation}
Equivalently, $\mu_{\bf k}$ has no atom at $0$.
\end{lem}

\begin{proof}
Let ${\bf p}=(k_1,\ldots,k_r)$.  If $r>1$, then ${\bf p}\succ(k_1)$, so Proposition \ref{strictMLR} shows that
$r(t)=\sigma_{\bf p}(t)/\sigma_{(k_1)}(t)$ is increasing.  If $\varphi$ is decreasing on $(0,1)$, then
\[
\int_0^1\varphi(t)\sigma_{\bf p}(t)dt
-\int_0^1\varphi(t)\sigma_{(k_1)}(t)dt
=\mathrm{Cov}_{\sigma_{(k_1)}dt}(\varphi,r)\leq 0.
\]
This follows, for example, by writing the covariance as one half of the double integral of
$(\varphi(t)-\varphi(s))(r(t)-r(s))$.
Taking $\varphi(t)=(1+xt)^{-1}$ gives
\[
0<H_{\bf p}(-x)\leq H_{(k_1)}(-x).
\]
Letting $r\rightarrow+\infty$ and using Lemma \ref{weakconvergence} gives the same inequality for $H_{\bf k}$.  The measure representing $H_{(k_1)}$ has a positive density on $(0,1)$ and no atom at $0$, so dominated convergence gives
$H_{(k_1)}(-x)\rightarrow0$.  This proves \eqref{Hminusinfty}.
\end{proof}

For an infinite index, equations \eqref{infiniteStieltjes} and \eqref{infinitebeta} give
\begin{equation}\label{Finfiniteint}
F_{\bf k}(z)
=z\int_{[0,1]}\frac{t\,d\mu_{\bf k}(t)}{1-tz},
\end{equation}
and
\begin{equation}\label{Fhatinfinite}
\widehat F_{\bf k}(z)
=\frac{F_{\bf k}(z)}{\beta({\bf k})z}
=\int_{[0,1]}\frac{t\,d\mu_{\bf k}(t)/\beta({\bf k})}{1-tz}
\in\mathfrak{T}.
\end{equation}
The representing probability measure in \eqref{Fhatinfinite} has no atom at $0$ (the factor $t$ annihilates any mass there) and has positive mass on $(0,1]$.  Therefore $\widehat F_{\bf k}$ has no zeros in $\Lambda$, and $F_{\bf k}$ has only the simple zero at $0$.

\begin{prop}\label{infiniteP}
For every infinite index ${\bf k}$,
\[
P_{\bf k}(z)=\frac{zF_{\bf k}'(z)}{F_{\bf k}(z)}
\]
belongs to $\mathfrak{T}$ and is not constant.  Hence, for $0<\theta<\pi$,
\begin{equation}\label{infiniteSelf}
\frac{d}{d\theta}\arg F_{\bf k}(e^{i\theta})>0,
\qquad
\frac{d}{d\theta}|F_{\bf k}(e^{i\theta})|<0.
\end{equation}
\end{prop}

\begin{proof}
By Lemma \ref{weakconvergence}, the functions and their derivatives attached to the finite prefixes converge locally uniformly on $\Lambda$.  To handle the common zero at the origin, write
\[
B_r(z)=\frac{F_{(k_1,\ldots,k_r)}(z)}z,
\qquad B(z)=\frac{F_{\bf k}(z)}z,
\]
with the removable values filled in at $0$.  Then $B_r$ and their derivatives converge locally uniformly to $B$, and \eqref{Fhatinfinite} shows that $B$ has no zero on $\Lambda$.  Hence
\[
P_{(k_1,\ldots,k_r)}(z)=1+z\frac{B_r'(z)}{B_r(z)}\quad
\rightarrow \quad1+z\frac{B'(z)}{B(z)}=P_{\bf k}(z)
\]
locally uniformly on $\Lambda$.  The class $\mathfrak{T}$ is closed under such limits, so $P_{\bf k}\in\mathfrak{T}$.  Its expansion begins with
\[
P_{\bf k}(z)=1+\frac{A_{\bf k}(2)}{A_{\bf k}(1)}z+O(z^2),
\]
and $A_{\bf k}(2)>0$, so it is not constant.  The proof of \eqref{selfmonotonicity}, with the representing measure of $P_{\bf k}$ in place of the finite one, gives \eqref{infiniteSelf}.
\end{proof}

\begin{prop}\label{infinitecross}
If ${\bf k}\succ{\boldsymbol\ell}$, then, for every $0<\theta<\pi$,
\[
|F_{\bf k}(e^{i\theta})|>|F_{\boldsymbol\ell}(e^{i\theta})|,
\qquad
\arg F_{\bf k}(e^{i\theta})>
\arg F_{\boldsymbol\ell}(e^{i\theta}).
\]
\end{prop}

\begin{proof}
Let
\[
c=\frac{\beta({\bf k})}{\beta({\boldsymbol\ell})}>1.
\]
For all sufficiently large $r$, the finite prefixes satisfy
$(k_1,\ldots,k_r)\succ(\ell_1,\ldots,\ell_r)$.  The functions
\[
q_r(z)=
\frac{\widehat F_{(k_1,\ldots,k_r)}(z)}
{\widehat F_{(\ell_1,\ldots,\ell_r)}(z)}
\]
belong to $\mathfrak{T}$ by \eqref{qfinite}.  Lemma \ref{weakconvergence} gives the locally uniform limit
\[
q(z)=\frac1c\frac{F_{\bf k}(z)}{F_{\boldsymbol\ell}(z)}
\in\mathfrak{T}.
\]
By Lemma \ref{noatomzero},
\[
q(-x)\quad\rightarrow\quad\frac1c.
\]
Thus the representing probability measure $\nu$ of $q$ has an atom of mass $1/c$ at $0$, and
\begin{equation}\label{strongquotientinfinite}
\frac{F_{\bf k}(z)}{F_{\boldsymbol\ell}(z)}
=1+c\int_{(0,1]}\frac{d\nu(t)}{1-tz},
\qquad
\nu((0,1])=1-\frac1c>0.
\end{equation}
The argument used after \eqref{strongquotientfinite} applies without change and proves both strict inequalities.
\end{proof}

\begin{Thm}\label{infiniteseparation}
Let ${\bf k}$ and ${\boldsymbol\ell}$ be distinct infinite indices.  Then
\[
H_{\bf k}(z_1)\neq H_{\boldsymbol\ell}(z_2)
\]
for all $|z_1|=|z_2|=1$ with $z_1,z_2\neq 1$.
For a fixed infinite index ${\bf k}$, the map
\[
z\mapsto H_{\bf k}(z)
\]
is one-to-one on $S^1\setminus\{1\}$.
\end{Thm}

\begin{proof}
After interchanging the indices if necessary, assume
${\bf k}\succ{\boldsymbol\ell}$.  Proposition \ref{infiniteP} and Proposition \ref{infinitecross} reproduce verbatim the two-case argument in the proof of Theorem \ref{finiteseparation}: on the upper semicircle, equality contradicts either strict increase of the argument or strict decrease of the modulus.  Conjugation handles the lower semicircle, opposite semicircles are separated by the sign of the imaginary part, and at $-1$ equation \eqref{strongquotientinfinite} gives a quotient strictly greater than $1$.  Since $H=1+F$, the same conclusion holds for $H$. For a fixed index, injectivity follows directly from the univalence of $H_{\bf k}$ on $\mathcal{D}$, because $S^1\setminus\{1\}\subset\mathcal{D}$.
\end{proof}

\section{The polylogarithm-star correspondence}\label{sec:correspondence}
We now turn the separation theorem into a global parametrization of the half-plane.  The radial parameter is supplied by the first moment.

\begin{prop}\label{binarycoordinate}
The map
\[
\beta:\widehat{\mathcal{T}}\rightarrow(0,1],
\qquad
{\bf k}\mapsto\sum_{j=1}^{+\infty}2^{-s_j({\bf k})}
\]
is an order-preserving bijection, where the index set is ordered as in Section \ref{sec:separation} and $(0,1]$ has its usual order.
Moreover,
\[
\begin{array}{lll}
\beta({\bf k})>\frac12\quad&\Leftrightarrow\quad &k_1=1,\\[2mm]
\beta({\bf k})=\frac12&\Leftrightarrow&{\bf k}=(2,\{1\}^\infty),\\[2mm]
\beta({\bf k})<\frac12&\Leftrightarrow&{\bf k}\in\mathcal{T}.
\end{array}
\]
\end{prop}

\begin{proof}
Associate with ${\bf k}$ the binary expansion whose digit in position $m$ is $1$ exactly when $m=s_j({\bf k})$ for some $j$.  Because the sequence of partial sums is infinite, this binary expansion contains infinitely many digits equal to $1$.  Conversely, every $x\in(0,1]$ has a unique binary expansion containing infinitely many $1$'s: at a dyadic rational one chooses the nonterminating expansion, and $1$ is written as $0.111\ldots$ in base $2$.  The successive gaps between the positions of the $1$'s give the positive integers $k_1,k_2,\ldots$.  This proves bijectivity.  Order preservation was proved in Lemma \ref{betamoment}. The first binary digit is $1$ precisely when $k_1=1$, which gives the first equivalence.  The nonterminating binary expansion of $1/2$ is $0.0111\ldots$, corresponding to $(2,\{1\}^\infty)$.  The last equivalence follows from the definition of $\mathcal{T}$.
\end{proof}

For $b\in(0,1]$, let ${\bf k}(b)$ be the unique index satisfying
$\beta({\bf k}(b))=b$.  We shall need continuity with respect to this symbolic coordinate.

\begin{prop}\label{indexcontinuity}
The map
\[
(b,z)\mapsto H_{{\bf k}(b)}(z)
\]
is jointly continuous on $(0,1]\times\mathcal{D}$, locally uniformly in $z$ on compact subsets of $\mathcal{D}$.
It extends continuously to the natural unit-circle parameter domain
\[
\Omega=\big((0,1]\times(S^1\setminus\{1\})\big)
\cup\big((0,\tfrac12)\times\{1\}\big).
\]
In particular, for $0<b_0<1/2$, the finite boundary value at $(b_0,1)$ is continuous under all approaches through $\Omega$.
\end{prop}

\begin{proof}
For a finite prefix ${\bf p}=(p_1,\ldots,p_r)$, put
\[
b_{\bf p}=\sum_{j=1}^{r}2^{-s_j({\bf p})}.
\]
The set of binary parameters of infinite extensions of ${\bf p}$ is the interval
\begin{equation}\label{cylinderinterval}
\left(b_{\bf p},b_{\bf p}+2^{-s_r({\bf p})}\right].
\end{equation}
Its diameter is at most $2^{-r}$.  By Lemma \ref{tailosc}, for every fixed $n$ the oscillation of
$A_{{\bf k}(b)}(n)$ on such a cylinder tends to zero uniformly as the prefix length tends to infinity.
If $b_0$ is not dyadic, then, for each prescribed prefix length, all sufficiently close $b$ have the same prefix as ${\bf k}(b_0)$.  The preceding oscillation estimate therefore gives
\[
A_{{\bf k}(b)}(n)\rightarrow A_{{\bf k}(b_0)}(n).
\]
Suppose that $b_0$ is dyadic.  If $b_0=1$, then parameters tending to $1$ share arbitrarily long initial strings of $1$'s with $(1,1,\ldots)$, and Lemma \ref{lim} proves continuity.  We may therefore assume $0<b_0<1$.  Let ${\bf p}=(p_1,\ldots,p_r)$ be its finite terminating index, so that
\[
b_0=\sum_{j=1}^{r}2^{-s_j({\bf p})}.
\]
The chosen nonterminating expansion is
\[
(p_1,\ldots,p_{r-1},p_r+1,1,1,\ldots),
\]
and Proposition \ref{speciallimits}, $(ii)$, identifies its function with $H_{\bf p}$.  Parameters approaching $b_0$ from below share arbitrarily long prefixes of this eventually-one index.  Parameters approaching from above have the form
\[
({\bf p},q,{\boldsymbol\gamma}),
\qquad q\rightarrow+\infty,
\]
with an arbitrary infinite tail ${\boldsymbol\gamma}$.  Put $m=n+1$.  Summing the arbitrary tail by Lemma \ref{lim} gives the quantitative bound
\begin{equation}\label{dyadicuniform}
0\leq A_{({\bf p},q,{\boldsymbol\gamma})}(n)-A_{\bf p}(n)
\leq A_{({\bf p},q-1)}(n)-A_{\bf p}(n)
\leq\sum_{j=2}^{m}j^{1-q}.
\end{equation}
The final finite sum tends to zero with $q$, uniformly in ${\boldsymbol\gamma}$.  Hence no jump occurs at a dyadic point.  We have proved that, for every $n$,
\[
b\mapsto A_{{\bf k}(b)}(n)
\]
is continuous on $(0,1]$.

Since $0\leq A_{{\bf k}(b)}(n)\leq1$, the series \eqref{Hinftyseries} converges uniformly in $b$ on every closed disc $|z|\leq r<1$.  This proves joint continuity there.  If $b_j\rightarrow b_0$, the family $H_{{\bf k}(b_j)}$ is locally uniformly bounded on $\mathcal{D}$ by the estimate in the proof of Proposition \ref{exi}, and it converges on the unit disc.  The Vitali--Porter theorem therefore gives locally uniform convergence on every compact subset of $\mathcal{D}$.

It remains to discuss $z=1$.  Fix $b_0<1/2$.  If $b_0$ is not dyadic, choose a finite prefix common to all sufficiently close parameters.  If its first component is at least $3$, then
\[
A_{{\bf k}(b)}(n)\leq(n+1)^{1-k_1},
\]
and the right-hand side is summable.  If its first component is $2$, choose a position $s$ in the common prefix at which $k_s\geq2$.  Summing every later tail by Lemma \ref{lim} gives the coefficientwise majorant
\[
A_{{\bf k}(b)}(n)
\leq A_{(k_1,\ldots,k_{s-1},k_s-1)}(n),
\]
whose sum is the finite value
$\zeta^{\star}(k_1,\ldots,k_{s-1},k_s-1)$.

If $b_0$ is dyadic, use the finite terminating index ${\bf p}$ above.  On the lower side, a sufficiently long common prefix of its nonterminating representative supplies one of the same majorants.  On the upper side, nearby indices are $({\bf p},q,{\boldsymbol\gamma})$ with $q$ large.  For a fixed $Q\geq2$ and $q\geq Q$, summing the arbitrary tail as in \eqref{dyadicuniform} gives
\[
A_{({\bf p},q,{\boldsymbol\gamma})}(n)
\leq A_{({\bf p},q-1)}(n)
\leq A_{({\bf p},Q-1)}(n).
\]
The last sequence is summable because it is attached to a convergent finite index, while \eqref{dyadicuniform} shows coefficientwise convergence to $A_{\bf p}(n)$.  Thus a common summable majorant exists on both sides.  Dominated convergence, now allowing $z\rightarrow1$ on the unit circle as well as $b\rightarrow b_0$, proves the last assertion.
\end{proof}

Define
\[
X=\overline{\mathcal{U}}\setminus
\left(\{0\}\cup\left[\frac12,1\right]\right).
\]
By Proposition \ref{binarycoordinate}, the map
\begin{equation}\label{taudef}
\tau_{\mathcal{S}}:\widetilde{\mathcal{ST}}\rightarrow X,
\qquad
\tau_{\mathcal{S}}(z,{\bf k})=z\beta({\bf k})
\end{equation}
is bijective.  With the topologies specified in Theorem \ref{one}, $\tau_{\mathcal{S}}$ is the usual polar-coordinate homeomorphism onto $X$.  For $u\in X$, write
\[
b=|u|,
\qquad z=\frac{u}{|u|},
\qquad {\bf k}={\bf k}(b),
\]
and define
\begin{equation}\label{Phidef}
\Phi_{\mathcal{S}}(u)=H_{{\bf k}(b)}(z).
\end{equation}
Thus
\[
\Phi_{\mathcal{S}}=\eta_{\mathcal{S}}\circ\tau_{\mathcal{S}}^{-1}.
\]
Proposition \ref{indexcontinuity} shows that $\Phi_{\mathcal{S}}$ is continuous.

\begin{lem}\label{halfplaneimage}
The map $\Phi_{\mathcal{S}}$ takes values in $\mathcal{H}\setminus\{1\}$.  More precisely, if $|u|<1$, then
\[
\mathrm{Re}\,\Phi_{\mathcal{S}}(u)>\frac12,
\]
whereas if $|u|=1$, then
\[
\mathrm{Re}\,\Phi_{\mathcal{S}}(u)=\frac12.
\]
\end{lem}

\begin{proof}
Let $z\in S^1\setminus\{1\}$.  From \eqref{infiniteStieltjes},
\begin{equation}\label{halfplanekernel}
\mathrm{Re}\frac1{1-tz}-\frac12
=\frac{1-t^2}{2|1-tz|^2}\geq0.
\end{equation}
Equality after integration is possible only when $\mu_{\bf k}=\delta_1$.  Since
\[
\int t\,d\mu_{\bf k}(t)=\beta({\bf k}),
\]
a probability measure on $[0,1]$ has first moment $1$ only when it is $\delta_1$.  Thus equality occurs precisely when $\beta({\bf k})=1$, that is, when ${\bf k}=(1,1,\ldots)$.  If $z=1$ is allowed, then
$H_{\bf k}(1)>1$.  This proves the two half-plane assertions. Finally, $H_{\bf k}(z)=1$ would imply $F_{\bf k}(z)=0$.  Equation \eqref{Fhatinfinite} shows that the only zero of $F_{\bf k}$ in $\Lambda$ is $0$, which is not on the unit circle.  At $z=1$ the value is greater than $1$.  Thus $1$ is never attained.
\end{proof}

\begin{prop}\label{Phiinjective}
The map $\Phi_{\mathcal{S}}:X\rightarrow\mathcal{H}\setminus\{1\}$ is injective.
\end{prop}

\begin{proof}
Suppose first that the two unit-circle parameters are different from $1$.  If the indices are different, Theorem \ref{infiniteseparation} separates the values.  If the indices are equal, univalence of $H_{\bf k}$ on $\mathcal{D}$ separates the two points.
If one parameter equals $1$, its value is a real number greater than $1$.  At every nonreal point of the unit circle the imaginary part of $H_{\bf k}$ is nonzero, by \eqref{Finfiniteint}.  At $z=-1$,
\[
H_{\bf k}(-1)=\int_{[0,1]}\frac{d\mu_{\bf k}(t)}{1+t}
\in\left[\frac12,1\right),
\]
so equality is again impossible.  If both parameters equal $1$, injectivity is exactly Theorem 1.3 of \cite{lit} (equivalently, the corresponding result of \cite{hmo}).
\end{proof}

\begin{prop}\label{Phiproper}
The map
\[
\Phi_{\mathcal{S}}:X\rightarrow\mathcal{H}\setminus\{1\}
\]
is proper.
\end{prop}

\begin{proof}
Let $u_j\in X$ approach the complement of $X$ in the closed unit disc.  We show that $\Phi_{\mathcal{S}}(u_j)$ cannot remain in a compact subset of
$\mathcal{H}\setminus\{1\}$. First suppose that $u_j\rightarrow0$.  Put $b_j=|u_j|$ and let $k_{1,j}$ be the first component of ${\bf k}(b_j)$.  Since
\[
2^{-k_{1,j}}\leq b_j\leq 2^{1-k_{1,j}},
\]
we have $k_{1,j}\rightarrow+\infty$.  Lemma \ref{lim} gives, for $m=n+1$,
\[
A_{{\bf k}(b_j)}(n)\leq m^{1-k_{1,j}}.
\]
Consequently, for all sufficiently large $j$ and uniformly for $|z|=1$,
\[
|F_{{\bf k}(b_j)}(z)|
\leq\sum_{m=2}^{+\infty}m^{1-k_{1,j}}
=\zeta(k_{1,j}-1)-1
\rightarrow0.
\]
Hence $\Phi(u_j)\rightarrow1$, which leaves every compact subset of the punctured target.

Now suppose that $u_j\rightarrow b_0$ for some $b_0\in[1/2,1]$.  Then $b_j=|u_j|\rightarrow b_0$ and the unit-circle factors
$z_j=u_j/|u_j|$ tend to $1$.  Let ${\bf k}_0={\bf k}(b_0)$, and put
${\bf e}=(2,\{1\}^{\infty})$.  Proposition \ref{speciallimits}, $(ii)$, gives
\[
H_{\bf e}(z)=H_{(1)}(z)=-\frac{\log(1-z)}z.
\]
Thus $|F_{\bf e}(e^{i\theta})|\rightarrow+\infty$ as
$\theta\rightarrow0$.  If $b_0>1/2$, then the first component of ${\bf k}_0$ is $1$, so
${\bf k}_0\succ{\bf e}$ and Proposition \ref{infinitecross} gives
\[
|F_{{\bf k}_0}(e^{i\theta})|>|F_{\bf e}(e^{i\theta})|.
\]
For $b_0=1/2$ the two functions are equal.  Therefore, in every case,
\begin{equation}\label{boundaryblowup}
|F_{{\bf k}_0}(e^{i\theta})|\rightarrow+\infty
\qquad(\theta\rightarrow0).
\end{equation}
Let $M>0$.  Choose $0<\theta_0<\pi$ so that the modulus in
\eqref{boundaryblowup} is greater than $2M$.  Proposition \ref{indexcontinuity}, applied at the fixed point $e^{i\theta_0}$, gives
\[
|F_{{\bf k}(b_j)}(e^{i\theta_0})|>M
\]
for all sufficiently large $j$.  If $\vartheta_j\in[0,\pi]$ is the absolute angular distance of $z_j$ from $1$, then
$\vartheta_j\rightarrow0$.  By Proposition \ref{infiniteP}, the modulus is strictly decreasing on the upper semicircle; by conjugation the same comparison holds below it.  Hence, for large $j$,
\[
|F_{{\bf k}(b_j)}(z_j)|
\geq |F_{{\bf k}(b_j)}(e^{i\theta_0})|>M.
\]
When $z_j=1$, the inequality follows by taking the limit from the upper semicircle.  Hence $|F_{{\bf k}(b_j)}(z_j)|\rightarrow+\infty$, and the triangle inequality
\[
|\Phi(u_j)|=|1+F_{{\bf k}(b_j)}(z_j)|
\geq |F_{{\bf k}(b_j)}(z_j)|-1
\]
shows that $|\Phi(u_j)|\rightarrow+\infty$ along the deleted slit.

To conclude, let $K$ be compact in $\mathcal{H}\setminus\{1\}$ and let $(u_j)$ be a sequence in $\Phi^{-1}(K)$.  A subsequence converges in the closed unit disc.  Its limit cannot be $0$ or a point of $[1/2,1]$ by the two alternatives just proved; hence the limit lies in $X$.  Continuity of $\Phi$ then keeps the limit in $\Phi^{-1}(K)$.  Thus $\Phi^{-1}(K)$ is sequentially compact, and therefore compact because $X$ is metrizable.  Hence $\Phi$ is proper.
\end{proof}

\begin{proof}[Proof of Theorem \ref{one}]
Let
\[
X^{o}=\bigg{\{}u\in\mathbb{C}\,\bigg|\,0<|u|<1\bigg{\}}
\setminus\left[\frac12,1\right),
\qquad
Y^{o}=\mathcal{H}^{o}\setminus\{1\}.
\]
These are open subsets of the plane, and $Y^{o}$ is connected.  By Lemma \ref{halfplaneimage} and Proposition \ref{Phiinjective}, the restriction
\[
\Phi_{\mathcal{S}}:X^{o}\rightarrow Y^{o}
\]
is continuous and injective.  It is also proper: if $K\subset Y^{o}$ is compact, then Proposition \ref{Phiproper} makes $\Phi^{-1}(K)$ compact in $X$, and Lemma \ref{halfplaneimage} shows that this preimage cannot meet $|u|=1$, so it lies in $X^{o}$.  The invariance of domain theorem (Theorem $62.3$ in \cite{munkres}) implies that $\Phi_{\mathcal{S}}(X^{o})$ is open in $Y^{o}$.  A proper continuous map between locally compact Hausdorff spaces is closed, so the image is also closed in $Y^{o}$.  Since the image is nonempty and $Y^{o}$ is connected,
\[
\Phi_{\mathcal{S}}(X^{o})=Y^{o}.
\]

It remains to identify the boundary.  If $|u|=1$, then
$\beta({\bf k})=1$ and ${\bf k}=(1,1,\ldots)$.  Proposition \ref{speciallimits}, $(i)$, gives
\[
\Phi_{\mathcal{S}}(u)=\frac1{1-u}.
\]
For $u=e^{i\theta}\neq1$,
\begin{equation}\label{cotformula}
\frac1{1-e^{i\theta}}
=\frac12+\frac{i}{2}\cot\frac{\theta}{2}.
\end{equation}
As $u$ runs once around $S^1\setminus\{1\}$, the right-hand side runs bijectively over the line
$\mathrm{Re}\,w=1/2$.  We have therefore proved that
\[
\Phi_{\mathcal{S}}:X\rightarrow\mathcal{H}\setminus\{1\}
\]
is bijective.  It is continuous and proper, and both spaces are locally compact Hausdorff spaces.  A proper continuous bijection between such spaces is a homeomorphism.  Recalling \eqref{taudef} and \eqref{Phidef} proves every assertion of Theorem \ref{one}.
\end{proof}

\begin{proof}[Proof of Corollary \ref{minusone}]
For every infinite index,
\[
H_{\bf k}(-1)=\int_{[0,1]}\frac{d\mu_{\bf k}(t)}{1+t}
\in\left[\frac12,1\right).
\]
The lower endpoint occurs exactly for $\mu_{\bf k}=\delta_1$, that is, for ${\bf k}=(1,1,\ldots)$.  The value $1$ would require $\mu_{\bf k}=\delta_0$, but this is impossible because its first moment is $\beta({\bf k})>0$.  Theorem \ref{one} shows that every real point in this interval has a unique preimage.  A nonreal unit-circle parameter gives a nonreal value, and the parameter $z=1$ gives a value greater than $1$.  Hence the unique preimage must have $z=-1$.  This proves the stated bijection.
\end{proof}

\section{Cyclotomic density and derived sets}\label{sec:cyclotomic}
We finish with the approximation consequence of the correspondence.

\begin{proof}[Proof of Theorem \ref{cyc}]
The set $\mathcal{C}^{\star}$ is countable.  We first show that it is contained in $\mathcal{H}^{o}$.  For $z\in S^1\setminus\{1\}$ and a finite index ${\bf k}$, equations \eqref{StieltjesH} and \eqref{halfplanekernel} give
\[
\mathrm{Re}\,H_{\bf k}(z)>\frac12,
\]
because the density $\sigma_{\bf k}$ is positive on $(0,1)$ and is not concentrated at $1$.  At $z=1$, the convergence condition in the definition of $\mathcal{C}^{\star}$ gives $k_1\geq2$, and
\[
H_{\bf k}(1)=\zeta^{\star}({\bf k})>1.
\]
Thus $\mathcal{C}^{\star}\subseteq\mathcal{H}^{o}$.

Let $w\in\mathcal{H}^{o}\setminus\{1\}$.  By Theorem \ref{one}, there is a unique pair $(z,{\bf k})\in\widetilde{\mathcal{ST}}$ such that
\[
w=H_{\bf k}(z).
\]
If $z\neq1$, choose roots of unity $z_r\neq1$ tending to $z$ and let
${\bf k}^{(r)}=(k_1,\ldots,k_r)$.  Choose a closed arc $K\subset S^1\setminus\{1\}$ containing $z$ and all sufficiently large $z_r$; then $K$ is a compact subset of $\mathcal{D}$.  Lemma \ref{weakconvergence} and continuity of $H_{\bf k}$ give
\[
\begin{split}
|H_{{\bf k}^{(r)}}(z_r)-H_{\bf k}(z)|
&\leq\sup_{\zeta\in K}|H_{{\bf k}^{(r)}}(\zeta)-H_{\bf k}(\zeta)|
  +|H_{\bf k}(z_r)-H_{\bf k}(z)|\rightarrow0.
\end{split}
\]
Every term on the left belongs to $\mathcal{C}^{\star}$.  If $z=1$, then ${\bf k}\in\mathcal{T}$ and, by the definition of the infinite zeta-star value,
\[
H_{(k_1,\ldots,k_r)}(1)
=\zeta^{\star}(k_1,\ldots,k_r)
\rightarrow H_{\bf k}(1)=w.
\]
Finally, the omitted point $1$ is also a limit point, because
\[
H_{(q)}(-1)=\sum_{m=1}^{+\infty}\frac{(-1)^{m-1}}{m^q}
\rightarrow1
\qquad(q\rightarrow+\infty).
\]
Hence $\mathcal{C}^{\star}$ is dense in $\mathcal{H}^{o}$.

Since $\mathcal{C}^{\star}\subseteq\mathcal{H}^{o}$ and
$\overline{\mathcal{H}^{o}}=\mathcal{H}$, one has
$(\mathcal{C}^{\star})'\subseteq\mathcal{H}$.  Conversely, every neighborhood of a point of $\mathcal{H}$ contains a nonempty open subset $V$ of $\mathcal{H}^{o}$.  The dense set $\mathcal{C}^{\star}$ meets $V$ infinitely often: otherwise, deleting the finitely many points of $V\cap\mathcal{C}^{\star}$ would leave a nonempty open subset of $V$ disjoint from $\mathcal{C}^{\star}$.  Thus every point of $\mathcal{H}$ is an accumulation point, and
\[
(\mathcal{C}^{\star})'=\mathcal{H}.
\]
The closed half-plane has no isolated points, so $\mathcal{H}'=\mathcal{H}$.  Induction gives
\[
(\mathcal{C}^{\star})^{(n)}=\mathcal{H},
\qquad n\geq1.
\]
This proves the theorem.
\end{proof}

\begin{rem}
The proof is constructive in the following sense.  Once the inverse image $(z,{\bf k})$ of a target point has been found, truncating ${\bf k}$ and replacing $z$ by nearby roots of unity produces an explicit sequence of cyclotomic multiple zeta-star values. \end{rem}

\section*{Acknowledgements}

 This project is  supported  by the National Natural Science Foundation of China (Grant No. 12571009) and the Natural Science Foundation of Hunan Province, China (Grant No. 2026JJ40003).     AI tool (ChatGPT 5.6 Sol) is used for language polishing and developing  the strategy of proof  of Theorem \ref{finiteseparation}. All mathematical statements and proofs  are verified by the author.  The author takes full responsibility for the paper.

\end{document}